\documentclass[review,10pt]{elsarticle}

\usepackage{amssymb}
\usepackage{mathrsfs}
\usepackage{amsfonts}
\usepackage{dsfont}
\usepackage{bm}
\usepackage{indentfirst}
\usepackage{colortbl,dcolumn}
\usepackage{hyperref}
\usepackage{color}
\usepackage{caption}
\usepackage{amsmath,amsthm}
\usepackage{psfrag}
\usepackage[english]{babel}
\usepackage{diagbox}  
\usepackage{lipsum}
\usepackage{enumitem}  
\usepackage{multirow}
\usepackage{nomencl} 
\usepackage{booktabs}
\usepackage{textcomp}
\usepackage{graphicx} 
\usepackage{float}

\newcommand\tabcaption{\def\@captype{table}\caption}
\newcommand\figcaption{\def\@captype{figure}\caption}
\usepackage{tikz,pgf}
\usepackage{pgfplots}
\pgfplotsset{compat=1.18}
\usetikzlibrary{calc, positioning, shapes.geometric}
\usepackage{graphicx}  
\graphicspath{{figs/}}
\usepackage{amsthm}

\usepackage[letterpaper,margin=1in,right=0.8in]{geometry}

\allowdisplaybreaks
\numberwithin{equation}{section}

\newcommand{\R}{\mathbb{R}}

\allowdisplaybreaks

\newtheorem{theorem}{Theorem}[section]
\newtheorem{lemma}{Lemma}[section]
\newtheorem{proposition}{Proposition}[section]
\newtheorem{corollary}{Corollary}[section]
\newtheorem{remark}{Remark}[section]
\newtheorem{example}{Example}[section]
\newtheorem{assumption}{Assumption}[section]

\begin{document}
\begin{frontmatter}

\title{Euler-type methods for L\'evy-driven McKean-Vlasov SDEs with super-linear coefficients: mean-square error analysis}
\author{Jingtao Zhu}
\ead{bingbll.zhu@gmail.com}

\author{Yuying Zhao\corref{cor1}}
\ead{zhaoyuying78@gmail.com}

\author{Siqing Gan\corref{cor1}}
\ead{sqgan@csu.edu.cn}

\cortext[cor1]{Corresponding authors}
\address{School of Mathematics and Statistics, HNP-LAMA, Central South University, Changsha 410083, China}

\begin{abstract} 
We develop and analyze a general class of Euler-type numerical schemes for L\'evy-driven McKean-Vlasov stochastic differential equations (SDEs), where the
drift, diffusion and jump coefficients grow super-linearly in the state variable.
These numerical schemes are derived by incorporating projections or nonlinear transformations into the classical Euler method, with the primary objective of establishing moment bounds for the numerical solutions.
This class of schemes includes the tanh-Euler, tamed-Euler and sine-Euler schemes as special cases. In contrast to existing approaches that rely on a coercivity condition (e.g., Assumption B-1 in Kumar et al., arXiv:2010.08585), the proposed schemes remove such a restrictive assumption. We provide a rigorous mean-square convergence analysis and establish that the proposed schemes achieve convergence rates arbitrarily close to
$\frac{1}{2}$ for the interacting particle systems associated with L\'evy-driven McKean-Vlasov SDEs.
Several numerical examples are presented to illustrate the convergence behavior and validate the theoretical results.
\end{abstract}


\end{frontmatter}
\section{Introduction}
McKean-Vlasov SDEs have found widespread applications across diverse fields, including biology and chemistry (e.g., chemotactic interactions \cite{keller1970initiation}) and neuroscience (e.g., the Hodgkin-Huxley model \cite{baladron2012mean}). 
Traditionally, McKean-Vlasov SDEs have been modeled with Gaussian noise as the driving process. However, in many real-world systems, especially in finance, physics and biology, the noise is non-Gaussian, and the system may exhibit sudden jumps or discontinuities. To accurately capture such behaviors, it is essential to incorporate (non-Gaussian) L\'evy-type perturbations \cite{tankov2003financial,applebaum2009levy,duan2015introduction,liu2023large}.

Let $\big(\Omega, \mathcal{F}, \mathbb{P}\big)$ be a probability space, and let $\{\mathcal{F}_t\}_{0 \leq t \leq \mathcal{T}}$ ($\mathcal{T} > 0$) denote the filtration generated by both the $m$-dimensional Brownian motion $\{W(t)\}_{0 \leq t \leq \mathcal{T}}$ and the Poisson random measure $p_{\varphi}(dv, dt)$, where $\mathcal{F}_t$ is the $\sigma$-algebra that captures all information up to time $t$ from both the Brownian motion and the Poisson process.
Furthermore, let $p_{\varphi}(dv, ds)$ be an $\{\mathcal{F}_t\}_{0 \leq t \leq \mathcal{T}}$-adapted Poisson measure defined on a space $(\mathcal{E}, \mathfrak{E}, \varphi)$, where $\mathcal{E} \subseteq \mathbb{R}^d \setminus \{0\}$. The Poisson process $p_{\varphi}(dv, dt)$ is assumed to be independent of the Brownian motion ${W(t)}$. The compensated version of this measure is given by $\tilde{p}_{\varphi}(dv, dt) := p_{\varphi}(dv, dt) - \varphi(dv) dt$.
We now investigate the following L\'evy-driven McKean-Vlasov SDEs: 
\begin{equation} 
\label{eq:MVSDES-problem}
\begin{aligned}
X(t) =&X_0+\int_0^t f\left(s,X(s), \mathscr{L}(X(s))\right) d s+\int_0^t g\left(s,X(s), \mathscr{L}(X(s))\right) d W(s) \\
\quad \quad & + \int_0^t \int_{\mathcal{E}}h\left(s,X(s), \mathscr{L}(X(s)),v\right) \tilde{p}_{\varphi}(dv,ds),
\end{aligned}
\end{equation}
for $t \in [0,\mathcal{T}]$, almost surely, where $\{\mathscr{L}(X(t))\}_{t \geq 0}$ denotes the law of $X(t)$. Moreover, $f \colon \mathbb{M} \to \R^d$, $g \colon \mathbb{M}  \to \R^{d \times m}$ and $h \colon \mathbb{M} \times \mathcal{E} \to \R^{d} $ are measurable, where
$\mathbb{M}:=[0, \mathcal{T}] \times \R^d \times \mathscr P_2(\R^d)$, as defined in Table \ref{table:notations}.

When the measure flow $\{\mathscr{L}(X(t))\}_{t \ge 0}$ is determined, the coefficients $f$, $g$ and $h$ depend exclusively on time and state variables. Consequently, the McKean-Vlasov SDEs \eqref{eq:MVSDES-problem} reduce to classical SDEs if $(h=0)$ or to a jump-diffusion SDEs if $(h\ne 0)$. To the best of our knowledge, significant research efforts have been devoted to developing and rigorously analyzing convergence of numerical schemes for such SDEs (e.g.,  \cite{chen2019mean,dareiotis2016tamed,deng2019generalized,Platen10numerical,kohatsu2010jump,kumar2017explicit,wang2010compensated,kumar2021milstein,yang2024strong,lei2023first,wang2023mean} and references therein).

Most McKean-Vlasov SDEs rarely admit explicit closed-form solutions, making numerical methods a necessary alternative for their analysis and approximation. Discretizing the McKean-Vlasov SDE encounters additional challenges, as the coefficients are distribution-dependent and the distribution $\mathscr{L}(X(s))$ must be discretized as well. A prevalent approach for approximating McKean-Vlasov SDEs is the stochastic particle method \cite{bossy1997stochastic}, which employs a large interacting particle system (IPS) to simulate the McKean-Vlasov SDEs. This method is based on the propagation of chaos (POC) principle, utilizing the empirical distribution of particles as an approximation of the true distribution $\mathscr{L}(X(s))$. Therefore, to obtain a fully discretized numerical solution, an efficient scheme for discretizing the IPS is essential. 

\textbf{Literature review and research gap.} 
McKean-Vlasov SDEs whose coefficients exhibit linear growth, the Euler-Maruyama (EM) scheme and its associated analysis have been further developed in \cite{bao2022approximations, ding2021euler, li2023strong}. However, the coefficients of the majority of McKean-Vlasov SDEs do not satisfy the condition of linear growth. Recall that when the coefficients of an SDE without distribution-dependent terms, and do not satisfy the global Lipschitz condition and exhibit super-linear growth,  the EM method fails to converge to the exact solution over a finite time interval in both the mean-square sense and the numerical weak sense \cite{hutzenthaler2011strong, Matting02, kumar2021explicit}. A similar divergence phenomenon is observed in McKean-Vlasov SDEs, often referred to as "particle corruption", which has been extensively studied in \cite{dos2022simulation}.

Extensive research has been dedicated to developing numerical methods for McKean-Vlasov SDEs to address these challenges in cases where the coefficients fail to satisfy global Lipschitz continuity. Numerical schemes such as the tamed Euler method \cite{dos2022simulation,liu2023tamed,liu2023thetamed} have been developed for drift coefficients with super-linear growth in the state variable and linear dependence on the measure. When diffusion term also exhibits super-linear growth, specialized numerical schemes like the tamed EM method \cite{kumar2022well}, the adaptive Euler method \cite{reisinger2022adaptive}, the modified Euler method \cite{JIAN2025112284}, the truncated Euler method \cite{yuanping2024explicit,guo2024convergence} and the split-step method \cite{chen2022flexible,chen2024euler,chen2025wellposedness} have been developed, each achieving convergence under specific conditions. Comprehensive analyses of these approaches are detailed in \cite{dos2022simulation,liu2023tamed,liu2023thetamed,kumar2022well,reisinger2022adaptive,JIAN2025112284,yuanping2024explicit,guo2024convergence,chen2022flexible,chen2024euler,chen2025wellposedness} and references therein.

Notably, numerical methods for L\'evy-driven McKean-Vlasov SDEs particularly those addressing super-linear growth conditions remain a relatively unexplored research area. In \cite{biswas2020well}, the authors propose a tamed Euler method for the IPS associated with L\'evy-driven McKean-Vlasov SDEs, deriving convergence rates under the condition that the drift, diffusion, and jump coefficients grow superlinearly concerning the states. The authors of \cite{tran2024infinite} propose an alternative tamed-adaptive Euler method for the IPS associated with L\'evy-driven McKean-Vlasov SDEs, in which the drift and diffusion exhibit super-linear growth, whereas the jump term remains linear. Under these assumptions, strong convergence is achieved for both finite and infinite time horizons.

\textbf{Contributions.}
Building on existing work on specific numerical methods \cite{biswas2020well}, we develop a 
unified Euler-type framework incorporating nonlinear transformations for L\'evy-driven McKean-Vlasov SDEs with super-linearly growing coefficients. This framework systematically includes the tanh-Euler, tamed-Euler and sine-Euler schemes as special cases,
enabling a rigorous mean-square error analysis for this family of methods.

In a nutshell, the highlight achievements of this work are described as follows.
\begin{itemize}

\item \textit{Moment bound analysis.}
To address the combined challenges of superlinear growth coefficients, distributional dependence and the lack of coercivity, we introduce a new analytical technique for establishing moment bounds for the proposed Euler-type schemes \eqref{eq:Modified-Euler-method}. Unlike existing approaches (e.g., \cite{kumar2022well, biswas2020well}) that depend on strong coercivity assumptions (such as Assumption B-1 in \cite{kumar2022well}), our proposed methods remove this restrictive condition.
Instead, we adopt Assumptions \ref{ass:Gamma-control-conditions} and \ref{ass:Coefficient-comparison-conditions-of-f}, which ensure that the transformation operator $\Gamma$ is appropriately bounded and scales with a negative power of the time step $\Delta t$. These assumptions enable us to derive uniform moment estimates critical for convergence analysis.

\item \textit{Convergence analysis.} Theorem \ref{thm:con-result-particle-scheme} establishes the convergence rates of the proposed Euler-type schemes \eqref{eq:Modified-Euler-method} for IPS. Specifically, for IPS associated with L\'evy-driven McKean–Vlasov SDEs, the proposed schemes achieve convergence rates arbitrarily close to $\frac{1}{2}$, without imposing linear growth constraints on the equation coefficients. 
Through the POC, we extend these results to the L\'evy-driven McKean-Vlasov SDEs in Corollary \ref{cor:convergence-theorem}, demonstrating the consistency between discrete and continuous systems.

\item \textit{Unified numerical framework.}
The proposed Euler-type schemes \eqref{eq:Modified-Euler-method} based on nonlinear transformations provide a unified theoretical foundation for constructing explicit numerical schemes under superlinear growth conditions. 
Notably, our main theorem offers a direct convergence analysis for this family of schemes, including the tanh-Euler, tamed-Euler and sine-Euler schemes, thereby avoiding the need for separate case-by-case analyses.
\end{itemize}

We note that the $l$-th moments of Poisson increments such as
$\int_0^t \int_{\mathcal{E}} p_{\varphi}(dv,ds)$  contribute at most 
$O(\Delta t)$, 
contrasting with the $O(\Delta t^{l/2})$ ($l \in \mathbb{N}$) behavior of Wiener processes (see Lemma \ref{lem:estimate-of-difference-of-numerical-solution}), where $\Delta t$ is the stepsize.
When combined with superlinear jump terms, this leads to reduced second-moment 
contributions, explaining the observed convergence rate limitation below 
$\frac{1}{2}$ in general cases (Lemma \ref{lem:Difference-of-coefficients}). 
Notably, the optimal $\frac{1}{2}$ rate becomes attainable under enhanced smoothness conditions.

\textbf{Paper organization.}
Apart from the introduction, the appendix contains technical proofs and a detailed mean-square error analysis for the proposed Euler-type schemes \eqref{eq:Modified-Euler-method}. The main body of this work is organized as follows: In Section \ref{sec:settings}, we outline the assumptions on the coefficients of L\'evy-driven McKean-Vlasov SDEs and provide results on the POC. Section \ref{sec:Modified-Euler-approximation-and-moment-bound} focuses on the construction of the proposed Euler-type schemes, establishing their uniform moment boundedness and presenting their convergence result. 
Numerical examples are provided to confirm the
previous findings in Section \ref{sec:Numerical-Experiments}.

\noindent
\textbf{List of Notations and Definitions.}
\begin{table}[ht] 
\centering
\begin{tabular}{| c | p{13.5cm} | }
\hline
\textbf{Notation} & \textbf{Definition} \\ \hline
\(\langle \cdot, \cdot \rangle\) & Euclidean inner product on \(\mathbb{R}^d\) \\ \hline
\(| \cdot |\) & Euclidean norm on \(\mathbb{R}^d\) \\ \hline
\(|A|\) & Trace norm of a matrix \(A \in \mathbb{R}^{d \times m}\), defined as \(|A| := \sqrt{\operatorname{trace}(A^T A)}\) \\ \hline
\((\Omega, \mathcal{F}, \{\mathcal{F}_t\}_{t \in [0, \mathcal{T}]}, \mathbb{P})\) & A filtered probability space \\ \hline
\(\mathbb{E}\) & Expectation under the probability measure \(\mathbb{P}\) \\ \hline
\(\mathbb{M}\) & $[0, \mathcal{T}] \times \R^d \times \mathscr P_2(\R^d)$ 
\\ \hline
\(a_1 \vee b_1\) & Maximum of \(a_1\) and \(b_1\) \\ \hline
\(a_1 \wedge b_1\) & Minimum of \(a_1\) and \(b_1\) \\ \hline
\(\delta_y\) & Dirac measure at point \(y\) \\ \hline
\(\mathscr{P}(\mathbb{R}^d)\) & Space of probability distributions over \(\mathbb{R}^d\) \\ \hline
\(\mathscr{P}_q(\mathbb{R}^d)\) & Subspace of \(\mathscr{P}(\mathbb{R}^d)\), defined as \(\mathscr{P}_q(\mathbb{R}^d) := \left\{\rho \in \mathscr{P}(\mathbb{R}^d) : \int_{\mathbb{R}^d} |y|^q \rho(dy) < \infty \right\}, q \geq 1\) \\ \hline
\(\mathbb{W}_q(\rho_1, \rho_2)\) & \(L^q\)-Wasserstein distance between two probability measures \(\rho_1, \rho_2 \in \mathscr{P}_q(\mathbb{R}^d)\) \\ \hline
 \multirow{2}{*}{\(\mathbb{W}_2^2(\rho_1, \rho_2)\)} & \(L^2\)-Wasserstein distance for random variables \(X\) and \(Y\) with distributions \(\rho_1 = \mathscr{L}_X\) and \(\rho_2 = \mathscr{L}_Y\), \(\mathbb{W}_2^2(\rho_1, \rho_2) \leq \mathbb{E}[|X - Y|^2]\) \\ \hline
\(C\) & A generic positive constant that may take different values in different contexts \\ \hline
\end{tabular}
\caption{List of notations}
\label{table:notations}
\end{table}

\section{Assumptions and settings}
\label{sec:settings}
\subsection{Assumptions on coefficients of L\'evy-driven McKean-Vlasov SDEs}
We outline the coefficient conditions for the L\'evy-driven McKean-Vlasov SDEs \eqref{eq:MVSDES-problem}. In the following, $\bar{p} \geq 1$ denotes a fixed positive constant.
The following assumptions are imposed to hold uniformly in $0 \leq t \leq \mathcal{T}$, $ y, \bar{y} \in \mathbb{R}^d$ and $\rho, \bar{\rho} \in \mathscr{P}_2\left(\mathbb{R}^d\right)$.

\begin{assumption}
\label{ass:inital-val-con}
$\mathbb{E}[|X_0|^{2\bar{p}}] < \infty.$
\end{assumption}

\begin{assumption}  \label{ass:Coupled-mono-condi}
The functions $f$, $g$ and $h$ satisfy coupled monotonicity condition:
\begin{gather*}
\begin{aligned}
2\big\langle y-\bar{y},f(t,y,\rho)-f(t,\bar{y},\bar{\rho})\big\rangle + \left|g(t,y,\rho)-g(t,\bar{y},\bar{\rho})\right|^2+ & \int_\mathcal{E}\left|h(t,y,\rho, v)-h(t,\bar{y},\bar{\rho},v)\right|^2 \varphi(d v)  \\
& \leq C\left(|y-\bar{y}|^2+\mathbb{W}_2^2(\rho, \bar{\rho})\right),
\end{aligned}
\end{gather*}
for some constant $C>0$.
\end{assumption}

\begin{assumption} \label{ass:Coercivity-condition-Chen} 
The functions $f$, $g$ and $h$ satisfy the following coercivity condition:
\begin{gather*}
\label{eq:mon-condi-jump-sde}
\begin{aligned}
2\bar{p}|y|^{2\bar{p}-2}  & \big\langle y, f(t,y,\rho)\big\rangle + \bar{p}(2\bar{p}-1)|y|^{2\bar{p}-2}|g(t,y,\rho)|^2   \\
 & + (1+(2\bar{p}-2)\theta)\int_{\mathcal{E}}|h(t,y,\rho,v)|^{2\bar{p}} \varphi(d v) \leq C\left(1+|y|^{2\bar{p}}+\mathbb{W}_{2}^{2\bar{p}}(\rho,\delta_{0})\right),
\end{aligned}
\end{gather*}
for some constants $C,\theta>0$.
\end{assumption}

\begin{assumption} \label{ass:Continuity—Conditions-f} 
The function $f(t, y, \rho)$ is uniformly continuous in $y$.  
\end{assumption}

Under Assumptions \ref{ass:inital-val-con}-\ref{ass:Coercivity-condition-Chen} with $\bar{p}=1$, and \ref{ass:Continuity—Conditions-f}, the existence and uniqueness of the solution to \eqref{eq:MVSDES-problem} are established in \cite[Theorem 2.1]{biswas2020well}. Additionally, applying the $\rm It\hat{o}$ formula and Assumption \ref{ass:Coercivity-condition-Chen} with $\bar{p} \geq 1$, it suffices to show that the $p$-th moment of $X(t)$ is bounded. In particular, one can find a constant 
$C>0$ satisfying
\[\sup_{0 \leq t \leq \mathcal{T}} \mathbb{E}\left[|X(t)|^{2p}\right] \leq C\left(1+\mathbb{E}\left[|X_{0}|^{2\bar{p}}\right]\right), \quad p \in [1, \bar{p}].\] 

\begin{assumption} \label{ass:poly-growth-coeff-a}
The function $f$ exhibits polynomial growth:
\begin{align*}
\left|f(t,y,\rho)-f(t,\bar{y},\bar{\rho})\right| &\le C \big(\left(1+|y|^{\gamma }+|\bar{y}|^{\gamma }\right)\left|y-\bar{y}\right|+\mathbb{W}_{2}(\rho,\bar{\rho})\big),
\end{align*}
for some constants $C,\gamma>0$.
\end{assumption}

\begin{assumption}  \label{ass:Initial-value-is-bounded}
There exists $C>0$ such that
\begin{align*}
\left|f(t,0,\delta_{0})\right|^{2} \vee \left|g(t,0,\delta_{0})\right|^{2} \vee \int_{\mathcal{E}} \left|h(t,0,\delta_{0},v)\right|^{2}\varphi(dv) \le C.
\end{align*}
\end{assumption}

\begin{assumption} \label{ass:holder-continuous}
    The functions $f,g$ and $h$ satisfy the $ H\ddot{o}lder$ continuity in time:
    \begin{align*}
        \left|f(t,y,\rho)-f(s,y,\rho)\right|+ \left|g(t,y,\rho)-g(s,y,\rho)\right|+ \int_{\mathcal{E}}\left|h(t,y,\rho,v)-h(s,y,\rho,v)\right|\varphi(dv) \le  C \left|t-s\right|^{\frac{1}{2}}.
    \end{align*}
    for some constant $C>0$.
\end{assumption}
The Assumptions \ref{ass:poly-growth-coeff-a} and \ref{ass:Initial-value-is-bounded} ensure that
the polynomial growth 
\begin{align} 
    \left|f(t,y,\rho)\right| 
    &~\le C \left(\left(1+|y|^{\gamma}\right)|y|+\mathbb{W}_{2}(\rho,\delta_{0})\right)  +C \notag \\
    &~\le C\big(1+|y|^{\gamma+1}+\mathbb{W}_{2}(\rho,\delta_{0})\big).\label{ineq:growth-condition-of-a}
\end{align}

In addition, one can apply the Cauchy-Schwarz inequality, Young's inequality, Assumptions \ref{ass:Coupled-mono-condi} and \ref{ass:poly-growth-coeff-a} to show
\begin{align} 
    \left|g(t,y,\rho)-g(t,\bar{y},\bar{\rho})\right|^{2} 
    & \leq C\left(|y-\bar{y}|^2+\mathbb{W}_2^2(\rho, \bar{\rho})\right) + 2\left|y-\bar{y}\right| \left|f(t,y,\rho)-f(t,\bar{y},\bar{\rho})\right| \notag \\
    &\le C \left(\left(1+|y|^{\gamma }+|\bar{y}|^{\gamma }\right)\left|y-\bar{y}\right|^{2}+\mathbb{W}_{2}^{2}(\rho,\bar{\rho})\right).
    \label{ineq:Polynomial-Growth-of-b}   
\end{align}
This implies the polynomial growth
\begin{equation}
\label{eq:poly-growth-diffusion}
\left|g(t,y,\rho)\right| \leq C\big(1+|y|^{\frac{\gamma}{2}+1}+\mathbb{W}_{2}(\rho,\delta_{0})\big).
\end{equation}

In a similar way, from Assumptions \ref{ass:Coupled-mono-condi} and \ref{ass:poly-growth-coeff-a}, we derive
\begin{align}
\int_{\mathcal{E}}\left|h(t,y,\rho,v)-h(t,\bar{y},\bar{\rho},v)\right|^{2} \varphi(dv) &\le C \left(\big(1+|y|^{\gamma }+|\bar{y}|^{\gamma }\big)\left|y-\bar{y}\right|^{2}+\mathbb{W}_{2}^{2}(\rho,\bar{\rho})\right). \label{ineq:Polynomial-Growth-of-c}   
\end{align}

Further, due to Assumption \ref{ass:Coercivity-condition-Chen}, Cauchy-Schwarz's inequality, \eqref{ineq:growth-condition-of-a} and the Young inequality, one can derive
\begin{align*}
\int_{\mathcal{E}} \left|h(t,y,\rho,v)\right|^{2\bar{p}}\varphi(dv) 
&~\le C\left(1+|y|^{2\bar{p}}+\mathbb{W}_{2}^{2\bar{p}}(\rho,\delta_{0})\right) + C |y|^{2\bar{p}-1} \big(1+|y|^{\gamma+1}+\mathbb{W}_{2}(\rho,\delta_{0})\big) \\
&~\le C\left(1+|y|^{\gamma+2\bar{p}}+\mathbb{W}_{2}^{2\bar{p}}(\rho,\delta_{0})\right).
\end{align*}
Also, for $ 2\le q\le 2\bar{p}$, thanks to the $\rm H\ddot{o}lder$ inequality  and $\varphi(\mathcal{E})< \infty $, we obtain 
\begin{align}
\int_{\mathcal{E}} \left|h(t,y,\rho,v)\right|^{q}\varphi(dv) 
&~\le \left(\int_{\mathcal{E}}\left|h(t,y,\rho,v)\right|^{q\cdot \frac{2\bar{p}}{q}} \varphi(dv) \right)^{\frac{q}{2\bar{p}}} \left(\int_{\mathcal{E}}\varphi(dv)\right)^{\frac{2\bar{p}-q}{2\bar{p}}}\notag \\
&~ \le C \left(C\left(1+|y|^{\gamma+2\bar{p}}+\mathbb{W}_{2}^{2\bar{p}}(\rho,\delta_{0})\right)\right)^{\frac{q}{2\bar{p}}} \notag \\
&~\le C\left(1+|y|^{\gamma+q}+\mathbb{W}_{2}^{q}(\rho,\delta_{0})\right).
\label{ineq:growth-condition-of-c}    
\end{align}

It is important to note that the functions in the L\'evy-driven McKean-Vlasov SDE \eqref{eq:MVSDES-problem} meet the $\mathbb{W}_{2}$-Lipschitz condition with respect to the measure component. From this point onward, the aforementioned assumptions will be systematically enforced. Additional examples that satisfy these assumptions are provided in Examples \ref{exam:3/2-volatility-model} and \ref{exam:double-well-model} in Section \ref{sec:Numerical-Experiments}.

\subsection{Particle approximation for the L\'evy-driven McKean-Vlasov SDEs}
Concerning their distribution dependence, L\'evy-driven McKean-Vlasov SDEs require the approximation of the measure component $\mathscr{L}(X(t))$ for all $t \geq 0$, which is commonly achieved through a stochastic IPS. 

For any fixed $N\in \mathbb{N}$, Let $\{W^{i}(t)\}_{t\ge 0}$, $\{\tilde{p}^{i}_{\varphi}(dv,dt)\}$ and $X_{0}^{i}$, for $i\in \mathcal{I}_{N}:=\{1,\cdots,N\}$, represent $N$ independent copies of $\{W(t)\}_{t\ge 0}$, $\{\tilde{p}_{\varphi}(dv,dt)\}$ and  $X_{0}^{i}$, respectively. 
The IPS is described by the following system:

\begin{equation} \label{eq:interat-parti-system}
\begin{aligned}
X^{i,N}(t)= &X_0^{i,N}+\int_0^t f\left(s, X^{i,N}(s), \rho_{s}^{X,N}\right) d s + \int_0^t g\left(s, X^{i,N}(s), \rho_{s}^{X,N}\right) d W^{i}(s)    \\
& +\int_0^t \int_{\mathcal{E}}h\left(s,X^{i,N}(s), \rho_{s}^{X,N},v\right) \tilde{p}_{\varphi}^{i}(dv,ds), \quad a.s.
\end{aligned}
\end{equation}
where $
\rho_{s}^{X,N}(\cdot) = \frac{1}{N} \sum_{i=1}^{N} \delta_{X^{i,N}(s)}(\cdot).
$ represents the empirical measure of $N$ interacting particles.

In the IPS \eqref{eq:interat-parti-system}, the particle
$X^{i,N}(t)$ provides a precise approximation of
$X(t)$ in the L\'evy-driven McKean-Vlasov SDEs with \eqref{eq:MVSDES-problem} when
$N$ is sufficiently large. This behavior is referred to as the propagation of chaos (POC). As a consequence of distributional dependence, the
$N$-dimensional system \eqref{eq:interat-parti-system} serves as an essential bridge for developing numerical approximations of L\'evy-driven McKean-Vlasov SDEs \eqref{eq:MVSDES-problem}. To characterize the POC, we now present the following system of non-interacting particles (NIPS):
\begin{equation} \label{eq:non-interat-parti-system}
\begin{aligned}
X^{i}(t)=&X_0^{i}+\int_0^t f\left(s,X^{i}(s), \mathscr{L}(X^{i}(s))\right) d s +\int_0^t g\left(s,X^{i}(s), \mathscr{L}(X^{i}(s))\right) d W^{i}(s) \\
&+\int_0^t \int_{\mathcal{E}}h\left(s, X^{i}(s), \mathscr{L}(X^{i}(s)),v\right) \tilde{p}_{\varphi}^{i}(dv,ds), \quad a.s.
\end{aligned}
\end{equation} 
If the  McKean-Vlasov SDEs with L\'evy noise \eqref{eq:MVSDES-problem} can be solved uniquely in the strong sense, then for $i \in \mathcal{I}_{N} $, the time-marginal distributions of $\{\mathscr{L}(X(t))\}_{t\ge 0}$ satisfy  
$\mathscr{L}(X(t)) = \mathscr{L}(X^{i}(t)).$

As demonstrated in \cite[Proposition 3.1]{biswas2020well}, the POC is ensured under Assumptions \ref{ass:inital-val-con}–\ref{ass:Continuity—Conditions-f}.
\begin{proposition}
\label{Pro:propa-of-chaos}
(POC, \cite[Proposition 3.1]{biswas2020well}) Let Assumptions \ref{ass:inital-val-con}-\ref{ass:Continuity—Conditions-f} be satisfied with $\bar{p}>2$. Then for some constant $ C > 0$ that does not depend on $d$ and $N$, we have the following estimation for arbitrary $N \in \mathbb{N}$,
$$
\sup _{i\in \mathcal{I}_{N}} \sup _{t \in[0, \mathcal{T}]} \mathbb{E}\left[\left|X^i(t)-X^{i, N}(t)\right|^2 \right] \leq C \begin{cases}N^{-1 / 2}, & \text { if } d<4, \\ N^{-1 / 2} \ln (N), & \text { if } d=4, \\ N^{-2 / d}, & \text { if } d>4. \end{cases}
$$
\end{proposition}
 
\section{Numerical schemes for the IPS associated with L\'evy-driven McKean-Vlasov SDEs} \label{sec:Modified-Euler-approximation-and-moment-bound}
We first partition the time interval $[0,\mathcal{T}]$ uniformly with step size $\Delta t = \frac{\mathcal{T}}{n}$, defining grid points $t_{k}=k\Delta t$ for $k=0,1,\cdots,n$.
For each particle $i\in \mathcal{I}_{N}$, the family of Euler-type schemes approximating the IPS \eqref{eq:interat-parti-system} is given by:
\begin{equation}  \label{eq:Modified-Euler-method}
\begin{aligned}
    Y_{t_{k+1}}^{i,N} = & Y_{t_{k}}^{i,N} + \Gamma_{1}\left(f\left(t_{k}, Y_{t_{k}}^{i,N},\rho_{t_{k}}^{Y,N}\right),\Delta t\right) \Delta t + \sum_{j=1}^{m} \Gamma_{2}\left(g_{j}\left(t_{k},Y_{t_{k}}^{i,N},\rho_{t_{k}}^{Y,N}\right),\Delta t \right) \Delta{W_{j}^{i}}(t_{k}) \\
    & + \int_{t_{k}}^{t_{k+1}}\int_{\mathcal{E}}\Gamma_{3}\left(h\left(t_{k},Y_{t_{k}}^{i,N},\rho_{t_{k}}^{Y,N},v\right),\Delta t \right)\tilde{p}_{\varphi}^{i}(dv,ds),    
\end{aligned}   
\end{equation}
 where the transformation operators $\Gamma_l(\cdot)$ $(l = 1,2,3)$ are approximations of '$\cdot$', $g_j  \colon \mathbb{M} \to \R^{d \times 1} $ is a vector function with dimension $d \times 1$, $\Delta W_{j}^{i}(t_{k}) = W_{j}^{i}(t_{k+1}) - W_{j}^{i}(t_{k})$ and $Y_{0}^{i,N}=X_{0}^{i}$.

\subsection{%
  \texorpdfstring{Assumptions for moment boundedness in the Euler-type schemes \eqref{eq:Modified-Euler-method}}
  {Assumptions for moment boundedness in the Euler-type schemes}%
}

\label{subsec:modified-Euler-schemes}
The definitions presented below serve to highlight the dependence of the scheme \eqref{eq:Modified-Euler-method} on the equation coefficients $f,g,h$ and the step size $\Delta t$. For any $0 \leq t \leq \mathcal{T}$, $ y, \bar{y} \in \mathbb{R}^d$, $\rho, \bar{\rho} \in \mathscr{P}_2\left(\mathbb{R}^d\right)$ and $v \in \mathcal{E}$, define $\bm{Y}:=(t,y,\rho),~\bm{\bar{Y}}:=(t,y,\rho,v)$. Let $F_1^{0}(\bm{Y})$, $F_2^{j}(\bm{Y})(j=1,\cdots,m)$ and $F_{3}^{m+1}(\bm{\bar{Y}})$ denote $f(\bm{Y})$, $g_{j}(\bm{Y})(j=1,\cdots,m)$ and $h(\bm{\bar{Y}})$, respectively.
Motivated by \cite{Zhang2017preserving,zhao2024one,JIAN2025112284}, we obtain the moment boundedness of Euler-type schemes \eqref{eq:Modified-Euler-method} under specific assumptions on the operators $\Gamma_{l}$, where $l=1,2,3$.
\begin{assumption} \label{ass:Gamma-control-conditions}
There exist $C, \alpha_{l}>0 ~(l=1,2,3)$ such that
\begin{align*}
\big|\Gamma_{l}\big(F_l^{\mathfrak{i}}(\cdot),\Delta t\big)\big| &\le C\Delta t^{-\alpha_{l}} \wedge |F_l^{\mathfrak{i}}(\cdot)|, \quad \mathfrak{i} =0,\cdots,m+1.
\end{align*} 
\end{assumption}

\begin{assumption} \label{ass:Coefficient-comparison-conditions-of-f}
There exist $ C, \hat{\delta},\hat{\gamma}>0 $ such that
\begin{align*}
\left|\Gamma_{1}(F_1^{0}(\bm{Y}),\Delta t)-F_1^{0}(\bm{Y})\right| \le & C \Delta t ^{\hat{\delta}} \left|F_1^{0}(\bm{Y})\right|^{\hat{\gamma}}.
\end{align*}
\end{assumption}
Assumptions \ref{ass:Gamma-control-conditions} and \ref{ass:Coefficient-comparison-conditions-of-f} provide fundamental conditions that control the mappings, ensuring boundedness in numerical schemes. Specifically, Assumption \ref{ass:Gamma-control-conditions} imposes linear growth constraints on these mappings and restricts their values using a negative power of the step size $\Delta t$. These restrictions are essential to prevent moment explosion, particularly when the drift, diffusion and jump coefficients exhibit polynomial growth. For instance, with the step size $\Delta t = 0.1$, the bounds $\big|\Gamma_{l}\big(F_l^{\mathfrak{i}}(\cdot),\Delta t\big)\big| \le C 10^{\alpha_{l}} \wedge |F_l^{\mathfrak{i}}(\cdot)|, l=1,2,3, \mathfrak{i}=0,\cdots, m+1,$ for some $C>0$, depend on specific mapping choices. 

In Table \ref{table:some-operators-list}, we provide different choices for the constraint operator $\Gamma_l(\cdot)$ $(l = 1,2,3)$ used in Euler-type schemes described in equation \eqref{eq:Modified-Euler-method}. Detailed specifications for the parameters involved in each operator are provided in section \ref{sec:examples}.

\begin{table}[htb]
\centering
\scalebox{1}{
\begin{tabular}{|c|c|} \hline 
\textbf{Example} & \textbf{Nonlinear transform operator} $\Gamma_l(z)$ $(l = 1,2,3)$ for Euler-type schemes \\ \hline  
Example \ref{ex:tanh} & $\Delta t^{-1} \tanh(\Delta t  z)$ \\ \hline
Example \ref{ex:tame} & $\frac{\mathrm{z}}{1 + \Delta t  |z|}$  \\ \hline
Example \ref{ex:sine} & $\Delta t^{-1} \sin(\Delta t  z)$  \\ \hline
\end{tabular}}
\caption{Choices for the nonlinear transform operator  \(\Gamma_l(\cdot)\) used in the Euler-type schemes \eqref{eq:Modified-Euler-method}.}
\label{table:some-operators-list}
\end{table}

\begin{remark}
Existing numerical methods for the McKean-Vlasov SDEs, such as those in \cite{kumar2022well,biswas2020well,biswas2025milsteintypeschemesmckeanvlasovsdes}, and the truncation approach in \cite{guo2024convergence}, ensure moment boundedness under alternative coercivity conditions. For example, there exist $C,\theta >0$ such that, 
\begin{gather}
\label{eq:mon-condi-jump-numerical-method}
\begin{aligned}
2\bar{p}|y|^{2\bar{p}-2}  & \big\langle y, \Gamma_{1}\big(f(t,y,\rho), \Delta t\big) \big\rangle +\sum_{j=1}^{m} \bar{p}(2\bar{p}-1)|y|^{2\bar{p}-2}  \Big|\Gamma_{2}\left(g_{j}\left(t,y,\rho\right),\Delta t \right) \Big|^2 \\
 & + (1+(2\bar{p}-2)\theta)\int_{\mathcal{E}} \Big|\Gamma_{3}\left(h\left(t,y,\rho,v\right),\Delta t \right)\Big|^{2\bar{p}} \varphi(d v) \leq C\left(1+|y|^{2\bar{p}}+\mathbb{W}_{2}^{2\bar{p}}(\rho,\delta_{0})\right).
\end{aligned}
\end{gather}
\end{remark}

Establishing the mean-square convergence rate of the proposed Euler-type schemes in equation \eqref{eq:Modified-Euler-method} relies on the boundedness of moments. 

\begin{lemma} \label{lem:moment-bounds-of-numerical-solution}
Let Assumptions \ref{ass:inital-val-con}, \ref{ass:Coercivity-condition-Chen}, \ref{ass:poly-growth-coeff-a}-\ref{ass:holder-continuous} and Assumptions  \ref{ass:Gamma-control-conditions}-\ref{ass:Coefficient-comparison-conditions-of-f}  be fulfilled. Then, for all $k=0,1,\cdots,n$, there exist $C,\beta>0$, such that 
\begin{align} \label{Moments-numer-bound-stro-con}
\sup_{i\in\mathcal{I}_{N}} \mathbb{E}\left[|Y_{t_{k}}^{i,N}|^{2p}\right] \leq C\left(1+\left(\mathbb{E}\left[\left|X_{0}^{i}\right|^{2\bar{p}}\right]\right)^{\beta}\right), \quad p\in \left[1,\frac{2\bar{p}-G}{2+2\bar{\alpha} G}\right],
\end{align}
where $G:=G\left(\gamma,\hat{\gamma},\hat{\delta} \right)=\frac{\hat{\gamma}(\gamma+1)-1}{\hat{\delta}} \vee 3\gamma $ and $\bar{\alpha} = \alpha_{1} \vee \left(\alpha_{2}+\frac{1}{2}\right) \vee \left(\alpha_{3}+1\right)$. Here $\gamma, \alpha_{1},\alpha_{2},\alpha_{3},\hat{\gamma},\hat{\delta}$ are defined by Assumptions \ref{ass:poly-growth-coeff-a}, \ref{ass:Gamma-control-conditions} and \ref{ass:Coefficient-comparison-conditions-of-f}, and $\bar{p}\ge 1+(\bar{\alpha}+\frac{1}{2})G(\gamma,\hat{\gamma},\hat{\delta})$.
\end{lemma}
The proof of this lemma is provided in detail in \ref{appen:proof-moment-bound-modi-method}. Equipped with bounded numerical moments, we can build up the convergence rates for numerical approximations.

\subsection{%
  \texorpdfstring{Mean-square convergence error analysis for the Euler-type schemes \eqref{eq:Modified-Euler-method}}
  {Mean-square convergence error analysis for the Euler-type schemes}%
}

\label{sec:Convergence-analysis}
In this section, we focus on the convergence error of the Euler-type schemes applied to the interacting particle system \eqref{eq:interat-parti-system}. Specifically, we consider the continuous-time version of the Euler-type schemes \eqref{eq:Modified-Euler-method} for this system. For any $i\in \mathcal{I}_{N}$, $s\in [0, \mathcal{T}]$, let $\tau_n(s) = \frac{\lfloor ns \rfloor}{n} = \sup \{ \eth \in \{t_k, k=0,1,\cdots,n\}, \eth  \leq s\}$, the continuous-time version of Euler-type schemes \eqref{eq:Modified-Euler-method} is given by:
\begin{equation}  \label{eq:continuous-Modified-Euler-method}
\begin{aligned}
    Y^{i,N}(s) = &~Y^{i,N}_{0} + \int_{0}^{s}\Gamma_{1}\left(f\left(\tau_{n}(r), Y^{i,N}\left(\tau_{n}(r)\right),\rho^{Y,N}_{\tau_{n}(r)}\right),\Delta t\right) dr  \\
    &~+ \sum_{j=1}^{m} \int_{0}^{s} \Gamma_{2}\left(g_{j}\left(\tau_{n}(r),Y^{i,N}\left(\tau_{n}(r)\right),\rho^{Y,N}_{\tau_{n}(r)}\right),\Delta t \right) d W_{j}^{i}(r) \\
    &~ + \int_{0}^{s}\int_{\mathcal{E}}\Gamma_{3}\left(h\left(\tau_{n}(r),Y^{i,N}\left(\tau_{n}(r)\right),\rho^{Y,N}_{\tau_{n}(r)},v\right),\Delta t \right)\tilde{p}^{i}_{\varphi}(dv,dr) .    
\end{aligned}   
\end{equation}

To characterize the differences between $f, g, h$ and $\Gamma_{1}, \Gamma_{2}, \Gamma_{3}$, we impose the following assumptions on Euler-type schemes \eqref{eq:Modified-Euler-method}. These assumptions hold for all $0 \leq t \leq \mathcal{T}$, $  y, \bar{y} \in \mathbb{R}^d$, $\rho,\bar{\rho}\in \mathscr{P}_{2}(\mathbb{R}^{d})$, and $v\in \mathcal{E}$, with the definitions of $F_{l}^{\mathfrak{i}}(l=1,2,3,~\mathfrak{i}=0,1,\cdots,m+1)$, $\bm{Y}$ and $\bm{\bar{Y}}$ as previously introduced.
\begin{assumption} \label{ass:Coefficient-comparison-conditions-of-Gamma1-Gamma3}
There exist $C>0, \delta_{l} \ge \frac{1}{2}$ and $\gamma_{l} \ge 1$ ($l=1,2,3$) such that
\begin{align*}
\left|\Gamma_{l}\left(F_l^{\mathfrak{i}}(\bm{Y}),\Delta t\right)-F_l^{\mathfrak{i}}(\bm{Y})\right| \le & C \Delta t ^{\delta_{l} } \left|F_l^{\mathfrak{i}}(\bm{Y})\right|^{\gamma_{l}},  \quad l=1,2, ~\mathfrak{i}=0,1,\cdots,m,  \\
\int_{\mathcal{E}}\left|\Gamma_{3}(F_{3}^{m+1}(\bm{\bar{Y}}),\Delta t)-F_{3}^{m+1}(\bm{\bar{Y}})\right|^{2}\varphi(dv) \le & C \Delta t ^{2\delta_{3} } \int_{\mathcal{E}}|F_{3}^{m+1}(\bm{\bar{Y}})|^{2\gamma_{3}}\varphi(dv).
\end{align*}
\end{assumption}

\begin{assumption}[\textbf{Enhanced coupled monotonicity condition}]
\label{ass:Enhanced-Coupled-mono-condi}
For some constants $C>0$ and $\eta>1$, we have
\begin{gather*}
\begin{aligned}
2\big\langle y-\bar{y},f(t,y,\rho)-f(t,\bar{y},\bar{\rho})\big\rangle + \eta \sum_{j=1}^{m}\left|g_{j}(t,y,\rho)-g_{j}(t,\bar{y},\bar{\rho})\right|^2+ &~\eta \int_\mathcal{E}\left|h(t,y, \rho, v)-h(t,\bar{y}, \bar{\rho},v)\right|^2 \varphi(d v)  \\
& \leq C\left(|y-\bar{y}|^2+\mathbb{W}_2^2(\rho, \bar{\rho})\right).
\end{aligned}
\end{gather*}
\end{assumption}

We note that different choices of $\Gamma_{l}$ for $l=1,2,3$ result in different values for the parameters $\delta_l$ and $\gamma_l$ ($l=1,2,3$). Some special choices for $\Gamma_l$  are presented in Section \ref{sec:examples}. In Assumption \ref{ass:Coefficient-comparison-conditions-of-f}, we only require that $\hat{\delta}$ and $\hat{\gamma}$ be greater than 0. However, in Assumption \ref{ass:Coefficient-comparison-conditions-of-Gamma1-Gamma3}, we impose the stronger conditions $\delta_1 \geq \frac{1}{2}$ and $\gamma_1 \geq 1$. Assumption \ref{ass:Enhanced-Coupled-mono-condi} is stronger than Assumption \ref{ass:Coupled-mono-condi} due to $\eta > 1$. Using similar arguments as those presented in \eqref{ineq:Polynomial-Growth-of-b}-\eqref{ineq:growth-condition-of-c}, we can establish polynomial growth for the diffusion and jump coefficients based on Assumptions \ref{ass:Enhanced-Coupled-mono-condi} and \ref{ass:poly-growth-coeff-a}.

Under these assumptions, we illustrate our conclusion: Euler-type schemes \eqref{eq:continuous-Modified-Euler-method} mean-square converges to the IPS \eqref{eq:interat-parti-system}.

\begin{theorem}
\label{thm:con-result-particle-scheme}
Let Assumptions \ref{ass:inital-val-con}, \ref{ass:Coercivity-condition-Chen},  \ref{ass:poly-growth-coeff-a}-\ref{ass:holder-continuous}, \ref{ass:Gamma-control-conditions}, and Assumptions \ref{ass:Coefficient-comparison-conditions-of-Gamma1-Gamma3}-\ref{ass:Enhanced-Coupled-mono-condi} hold. Then for any $\varepsilon>0$, there exists  $C$ independent of $n$ and $N$ such that the mean-square error between Euler-type schemes \eqref{eq:continuous-Modified-Euler-method} and the solution of the IPS \eqref{eq:interat-parti-system} satisfies
$$
\sup _{i\in \mathcal{I}_{N}} \sup _{t \in[0, \mathcal{T}]} \mathbb{E}\left[\left|X^{i, N}(t)-Y^{i, N}(t)\right|^{2}\right] \leq C \Delta t^{\frac{2}{2+\varepsilon}} \left(1+\left(\mathbb{E}\left[\left|Y_{0}^{i,N}\right|^{2\bar{p}}\right]\right)^{\beta}\right),
$$
where $\bar{p}$ comes from Assumption \ref{ass:Coercivity-condition-Chen} and $\beta>0$ is from Lemma \ref{lem:moment-bounds-of-numerical-solution}.
\end{theorem}

To ensure readability, we skip the proof and later present it in \ref{sec:proof}. For completeness, we quantify the numerical errors associated with Euler-type schemes \eqref{eq:Modified-Euler-method} when approximating L\'evy-driven McKean-Vlasov SDEs \eqref{eq:MVSDES-problem}. This result directly follows from the combination of Theorem \ref{thm:con-result-particle-scheme} and Proposition \ref{Pro:propa-of-chaos}, assuming bounded higher-order moments for the initial values.
\begin{corollary}  \label{cor:convergence-theorem}
    Under the same assumptions as in Theorem \ref{thm:con-result-particle-scheme}, for any $\varepsilon>0$, there exists  $C$ independent of $n$ and $N$ such that the mean-square error between Euler-type schemes \eqref{eq:continuous-Modified-Euler-method} and the solution of the NIPS \eqref{eq:non-interat-parti-system} satisfies
    $$
    \sup _{i\in \mathcal{I}_{N}} \sup _{t \in[0, \mathcal{T}]} \mathbb{E}\left[\left|X^i(t)-Y^{i,N}(t)\right|^2\right] \leq C \begin{cases}N^{-1 / 2}+\Delta t^{\frac{2}{2+\epsilon}}, & \text { if } d<4, \\ N^{-1 / 2} \ln (N)+\Delta t^{\frac{2}{2+\epsilon}}, & \text { if } d=4, \\ N^{-2 / d}+\Delta t^{\frac{2}{2+\epsilon}}, & \text { if } d>4. \end{cases}    $$
\end{corollary}

\subsection{%
  \texorpdfstring{Some choices of the operators in Euler-type schemes \eqref{eq:Modified-Euler-method}}
  {Some choices of the operators in Euler-type schemes}%
}

\label{sec:examples}
In this section, we explicitly present the operators $\Gamma_{i}(i=1,2,3) $. For the definitions of map $F_{l}^{\mathfrak{i}}$ and variables $\bm{Y}$ and $\bm{\bar{Y}}$, we refer to Section \ref{subsec:modified-Euler-schemes}.

\begin{example}[Tanh Euler method (TanhEM)]  \label{ex:tanh}
In Euler-type schemes \eqref{eq:Modified-Euler-method}, we focus on
\begin{equation}
    \begin{aligned} \label{ex:tanh-euler-scheme}
    \Gamma_{l}\left(F_{l}^{\mathfrak{i}}\left(\bm{Y}\right),\Delta t\right)& =  \Delta t^{-1}\tanh\left(\Delta t~F_{l}^{\mathfrak{i}}\left(\bm{Y}\right)\right), \quad l=1,2, ~\mathfrak{i}=0,1,\cdots,m,\\ 
     \Gamma_{3}\left(F_{3}^{m+1}\left(\bar{\bm{Y}}\right),\Delta t\right) & =  \Delta t^{-1} \tanh\left(\Delta t~F_{3}^{m+1}\left(\bar{\bm{Y}}\right)\right). 
    \end{aligned}
\end{equation}
We verify in \ref{verfivation-examples} that $\Gamma_{1},\Gamma_{2}$ and $\Gamma_{3}$ satisfy Assumptions \ref{ass:Gamma-control-conditions} and \ref{ass:Coefficient-comparison-conditions-of-Gamma1-Gamma3} with the parameters $\alpha_{l}=1$, $\delta_{l}=1$ and $\gamma_{l}=2$ for $l=1,2,3$.
\end{example}

\begin{example}[Tamed Euler method (TameEM)] \label{ex:tame}
In Euler-type schemes \eqref{eq:Modified-Euler-method}, we consider 
\begin{equation}
\begin{aligned} \label{ex:tamed-euler-scheme}
    \Gamma_{l}\left(F_{l}^{\mathfrak{i}}\left(\bm{Y}\right),\Delta t\right) & =  \frac{F_{l}^{\mathfrak{i}}\left(\bm{Y}\right)}{1+\Delta t\left|F_{l}^{\mathfrak{i}}\left(\bm{Y}\right)\right|}, \quad l=1,2, ~\mathfrak{i}=0,1,\cdots,m,\\
    \Gamma_{3}\left(F_{3}^{m+1}\left(\bar{\bm{Y}}\right),\Delta t\right) & =  \frac{F_{3}^{m+1}\left(\bar{\bm{Y}}\right)}{1+\Delta t\left|F_{3}^{m+1}\left(\bar{\bm{Y}}\right)\right|}.
\end{aligned}
\end{equation}
In \ref{verfivation-examples}, we verify that $\Gamma_{1},\Gamma_{2}$ and $\Gamma_{3}$ satisfy Assumptions \ref{ass:Gamma-control-conditions} and \ref{ass:Coefficient-comparison-conditions-of-Gamma1-Gamma3}, given the parameters $\alpha_{l}=1$, $\delta_{l}=1$ and $\gamma_{l}=2$ for $l=1,2,3$.
\end{example}

\begin{example}[Sine Euler method (SineEM)] \label{ex:sine}
In Euler-type schemes \eqref{eq:Modified-Euler-method}, we investigate
\begin{equation}
    \begin{aligned}  \label{ex:sine-euler-scheme}
    \Gamma_{l}\left(F_{l}^{\mathfrak{i}}\left(\bm{Y}\right),\Delta t\right)& =  \Delta t^{-1}\sin\left(\Delta t~F_{l}^{\mathfrak{i}}\left(\bm{Y}\right)\right), \quad l=1,2, ~\mathfrak{i}=0,1,\cdots,m,\\ 
     \Gamma_{3}\left(F_{3}^{m+1}\left(\bar{\bm{Y}}\right),\Delta t\right) & =  \Delta t^{-1} \sin\left(\Delta t~F_{3}^{m+1}\left(\bar{\bm{Y}}\right)\right). 
    \end{aligned}
\end{equation}
By a verification procedure similar to that in Example \ref{ex:tanh}, one then observes that Assumptions \ref{ass:Gamma-control-conditions} and \ref{ass:Coefficient-comparison-conditions-of-Gamma1-Gamma3} hold for $\alpha_{l}=1$, $\delta_{l}=1$, and $\gamma_{l}=2$, with $l=1,2,3$.
\end{example}

\begin{example}[Mix Euler method (MixEM)] \label{ex:Mix}
In Euler-type schemes \eqref{eq:Modified-Euler-method}, we discuss 
\begin{equation}
    \begin{aligned}   \label{ex:Mix-euler-scheme}
        \Gamma_{1}\left(F_{1}^{0}\left(\bm{Y}\right),\Delta t\right) & =  \Delta t^{-1}\sin\left(\Delta t~F_{1}^{0}\left(\bm{Y}\right)\right),  \\
        \Gamma_{2}\left(F_{2}^{\mathfrak{i}}\left(\bm{Y}\right),\Delta t\right)&=\frac{F_{2}^{\mathfrak{i}}\left(\bm{Y}\right)}{1+\Delta t\left|F_{2}^{j}\left(\bm{Y}\right)\right|}, 
    \quad \mathfrak{i}= 1, \cdots m, \\
        \Gamma_{3}\left(F_{3}^{m+1}\left(\bm{\bar{Y}}\right),\Delta t\right) & = \Delta t^{-1} \tanh\left(\Delta t~F_{3}^{m+1}\left(\bm{\bar{Y}}\right)\right).
    \end{aligned}
\end{equation}
Using a verification procedure similar to that in Examples \ref{ex:tanh}-\ref{ex:sine}, one observes that Assumptions \ref{ass:Gamma-control-conditions} and \ref{ass:Coefficient-comparison-conditions-of-Gamma1-Gamma3} are satisfied for $\alpha_{l}=1$, $\delta_{l}=1$, and $\gamma_{l}=2$, where $l=1,2,3$.
\end{example}
From Theorem \ref{thm:con-result-particle-scheme}, it follows that for any sufficient small $\varepsilon$, the mean-square convergence rates are $\frac{1-\varepsilon}{2}$ for Euler-type schemes in Examples \ref{ex:tanh}-\ref{ex:Mix}. We provide numerical evidence (Examples \ref{exam:3/2-volatility-model} and \ref{exam:double-well-model}) supporting the theoretical strong convergence rate.

\section{Numerical results} \label{sec:Numerical-Experiments}
We present numerical illustrations to support our theoretical findings by testing Euler-type schemes \eqref{eq:Modified-Euler-method} on two jump-extended models, in which all coefficients may exhibit superlinear growth with respect to the state variable. To examine the convergence rates of the numerical methods described in Section 4, we focus on TanhEM \eqref{ex:tanh-euler-scheme}, TameEM \eqref{ex:tamed-euler-scheme}, SineEM \eqref{ex:sine-euler-scheme}, and MixEM \eqref{ex:Mix-euler-scheme}.

In Examples \ref{exam:3/2-volatility-model} and \ref{exam:double-well-model}, the particle method is employed to approximate the law $\mathscr{L}(X(t_{k}))$ at each time step $t_k(k=0,1, \ldots, n)$, using its empirical distribution with $N=500$ particles. For the mean-square error (MSE), we approximate it at terminal time $\mathcal{T}$ in the following form:
$$
{\rm MSE} = \left(\frac{1}{N}\sum_{i=1}^{N}\left|Y_{\mathcal{T}}^{i,N,n_{\delta t}}-Y_{\mathcal{T}}^{i,N,n_{\Delta t}} \right|^{2}\right)^{\frac{1}{2}},
$$
where $Y_{\mathcal{T}}^{i,N,n_{\delta t}}$, $Y_{\mathcal{T}}^{i,N,n_{\Delta t}}$ represent the numerical solutions for the $i$-th particle with step sizes $\Delta t$ and $\delta t$, respectively. 

All numerical experiments were conducted in MATLAB R2022a on a laptop with a 12th Gen Intel\textsuperscript{\tiny\textregistered} Core\textsuperscript{\tiny\texttrademark} i5-12500H mobile processor (2.50 GHz base frequency), 16 GB DDR4 RAM, and Windows 11 Pro 64-bit operating system. To ensure reproducibility, the Mersenne Twister random number generator (algorithm ID: 'twister' in MATLAB) was initialized with a seed value of 100 using the rng(seed, 'twister') function prior to each simulation.

\begin{example} \label{exam:3/2-volatility-model}
Consider the jump-extended $\frac{3}{2}$-volatility model of the form
\begin{equation} \label{exam:Num-3/2-Volatility}
\left\{\begin{array}{l}
d V(t)= a_{1} \left(V(t)\left(a_{2}-|V(t)| \right)+ \mathbb{E}\left[V(t)\right]\right) d t + b \left(|V(t)|^{\frac{3}{2}}+\mathbb{E}\left[V(t)\right]\right) d W(t),  \\
\qquad \qquad +c\big(1-V(t)-\mathbb{E}\left[V(t)\right]\big)d\tilde{N}(t) \\
V_0=v,
\end{array}\right.
\end{equation}
where $t\in[0,1]$, $a_{1}=6, a_{2}=2, b= -0.1, c = 1$, $v=0.5$ and $\tilde{N}(t)$ denote a compensated Poisson process with jump intensity $\lambda = 2$.
\end{example}

We confirm that this model satisfies Assumptions \ref{ass:Coercivity-condition-Chen} and \ref{ass:poly-growth-coeff-a} with parameters $ \gamma =1, \bar{p}=297$ while setting $\eta=1.5, \theta =1$. A detailed verification is provided in \ref{verfivation-numerical-examples}.

To rigorously determine the mean-square convergence rates, we compute reference solutions using Monte Carlo simulations with sufficiently small time-step $\delta t = 1/2^{13}$ and $L = 2 \times 10^3$ independent sample paths. Figure \ref{fig:3/2 volatility} shows the mean-square error (MSE) of four numerical methods-TanhEM \eqref{ex:tanh-euler-scheme}, TameEM \eqref{ex:tamed-euler-scheme}, SineEM \eqref{ex:sine-euler-scheme}, and MixEM \eqref{ex:Mix-euler-scheme}-visualized for five separate step sizes $\Delta t = 2^{-l}(l=7,\cdots,11)$ on a log-log scale. We also include a dashed line with $\frac{1}{2}$ slope to visualize the contrast. The results indicate that the error curves for the TanhEM \eqref{ex:tanh-euler-scheme}, TameEM \eqref{ex:tamed-euler-scheme}, SineEM \eqref{ex:sine-euler-scheme} and MixEM \eqref{ex:Mix-euler-scheme} methods are nearly parallel to the reference line, confirming an empirical mean-square convergence rate close to $\frac{1}{2}$, consistent with theoretical expectations.
 
\begin{figure}[H]
\begin{center}
\includegraphics[height=5cm]{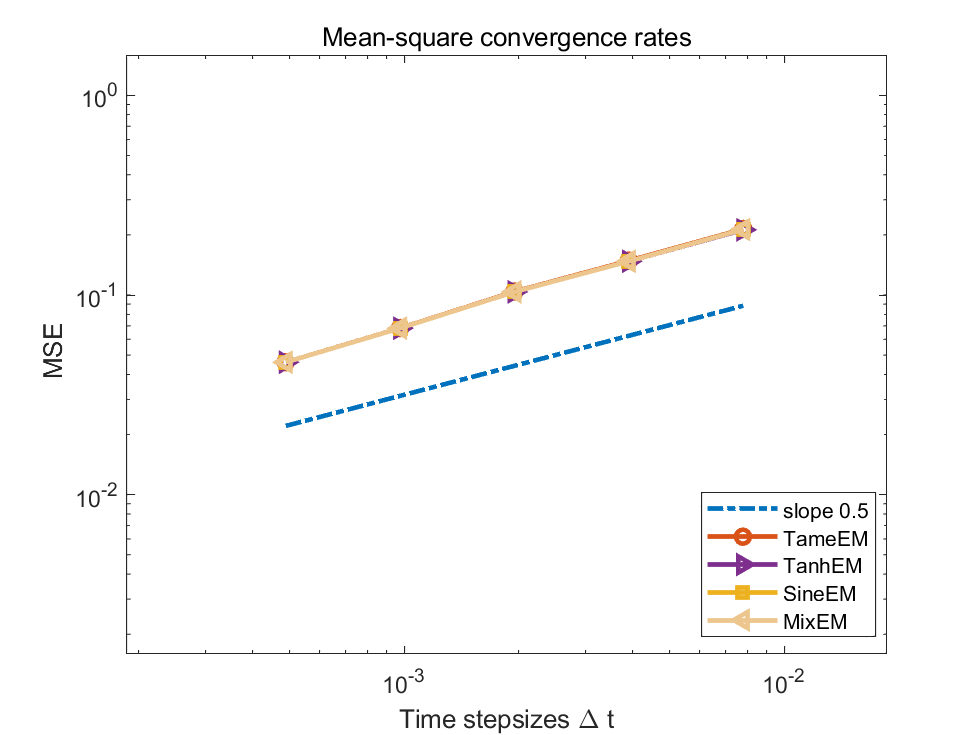}     
\caption{Convergence rates of Euler-type schemes for Example \ref{exam:3/2-volatility-model}}  
\label{fig:3/2 volatility}                   
\end{center}         
\end{figure}

In addition, $500$ simulated paths of TanhEM \eqref{ex:tanh-euler-scheme}, TameEM \eqref{ex:tamed-euler-scheme}, SineEM \eqref{ex:sine-euler-scheme} and MixEM \eqref{ex:Mix-euler-scheme} with step size $\Delta t = 2^{-8}$ are presented in Figures \ref{paths-TameEM-TanhEm-vola-model} and \ref{paths-SineEM-MixEm-vola-model}. It can be observed that the trajectories of the four different types of numerical methods on the interval $[0,1]$ are highly consistent, which illustrates the feasibility of our numerical methods. 

In the second example, all coefficients of are not globally Lipschitz in the state variable. We investigated the mean-square convergence rates of TanhEM \eqref{ex:tanh-euler-scheme}, TameEM \eqref{ex:tamed-euler-scheme}, SineEM \eqref{ex:sine-euler-scheme}, and MixEM \eqref{ex:Mix-euler-scheme} in this setting.
\begin{example} \label{exam:double-well-model}
Consider the jump-extended double well dynamics in the form of
\begin{equation} 
\left\{\begin{array}{l}  \label{exam:num-double-well model}
d w(t)=  d_{1}\left( w(t)\left(1-w(t)^{2}\right)+\mathbb{E}\left[w(t)\right]\right) d t + d_{2} \left(1-w(t)^{2}-\mathbb{E}\left[w(t)\right]\right) d W(t) \\
 \qquad \quad + d_{3} \big(w(t)\ln\left(1+w(t)^{2}\right)+\mathbb{E}\left[w(t)\right]\big)d\tilde{N}(t),  \\
w_0=\mu,
\end{array}\right.
\end{equation}
where $t\in[0,1]$, $d_{1}=66, d_{2}=0.19, d_{3} =0.0006$, $\mu=0.5$ and $\tilde{N}(t)$ as before.
\end{example}

\begin{figure}[H]
\begin{minipage}[t]{0.5\linewidth} 
\centering
\includegraphics[width=7cm,height=4cm]{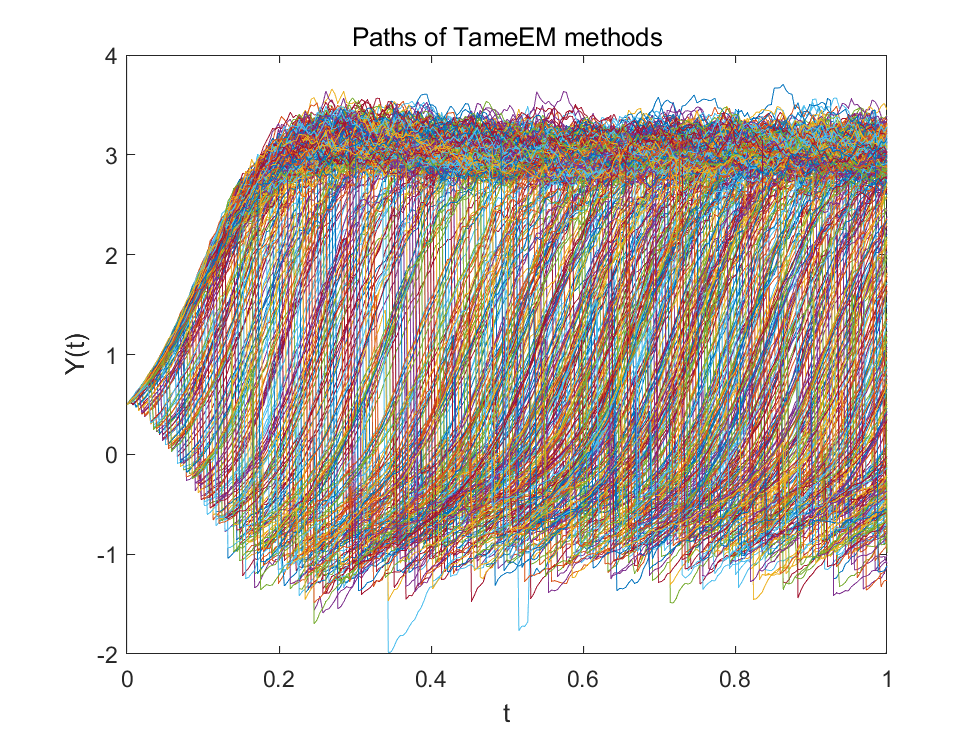}   
\end{minipage}
\hfill
\begin{minipage}[t]{0.5\linewidth}  
\centering
\includegraphics[width=7cm,height=4cm]{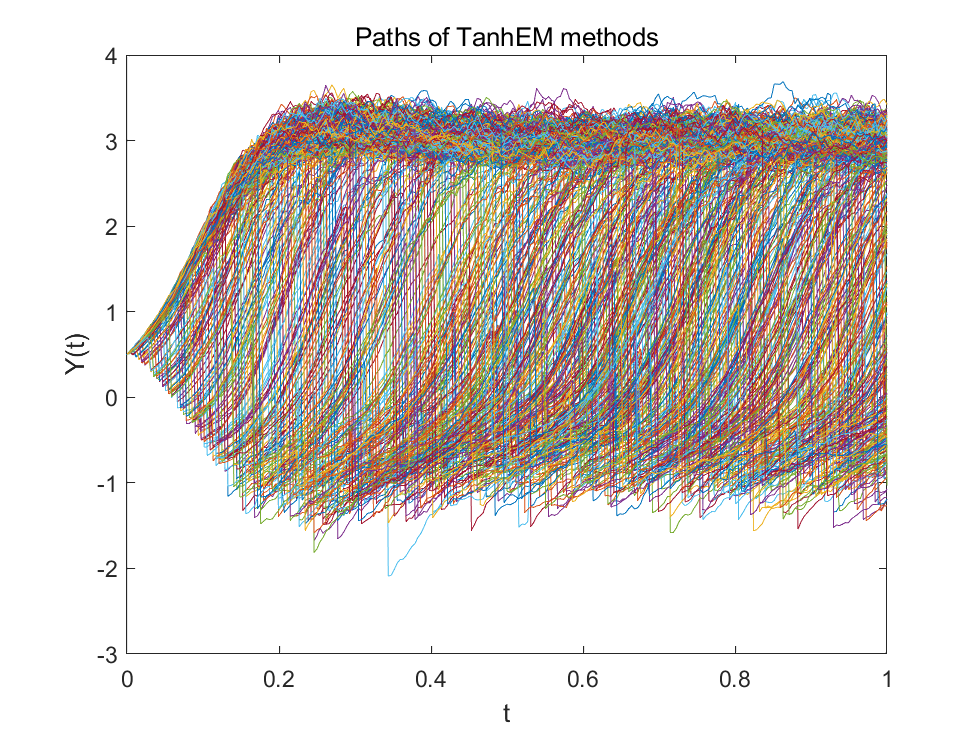} 
\end{minipage}
\caption{Paths of TameEM (left) and TanhEM methods (right) for Example \ref{exam:3/2-volatility-model}}
\label{paths-TameEM-TanhEm-vola-model}
\end{figure}
\begin{figure}[H]
\begin{minipage}[t]{0.5\linewidth} 
\centering
\includegraphics[width=7cm,height=4cm]{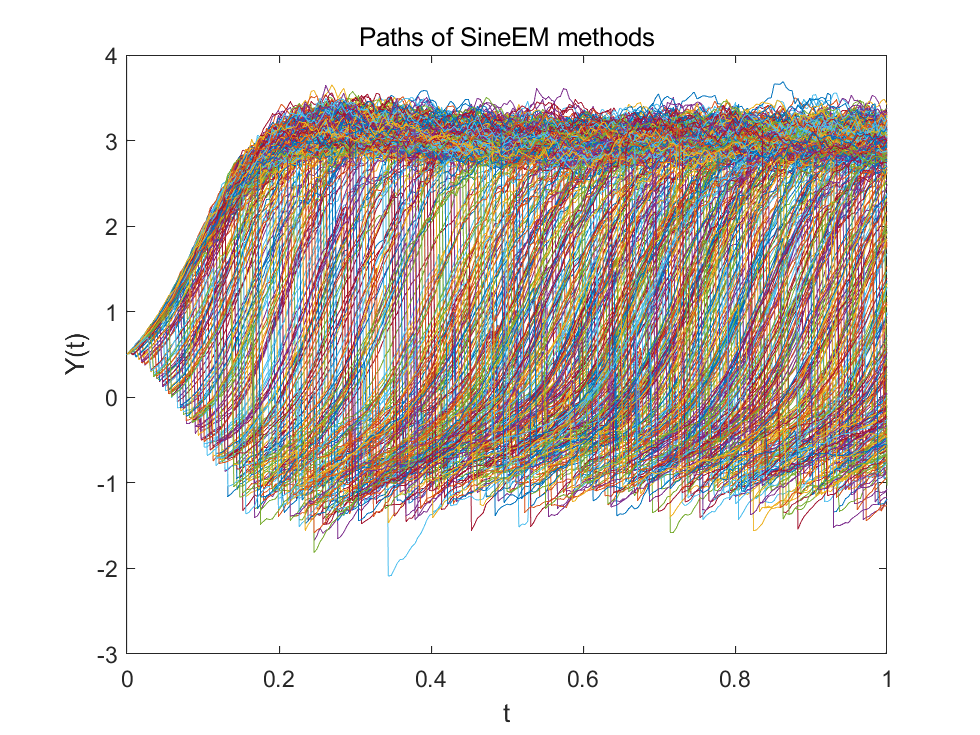}   
\end{minipage}
\hfill
\begin{minipage}[t]{0.5\linewidth}  
\centering
\includegraphics[width=7cm,height=4cm]{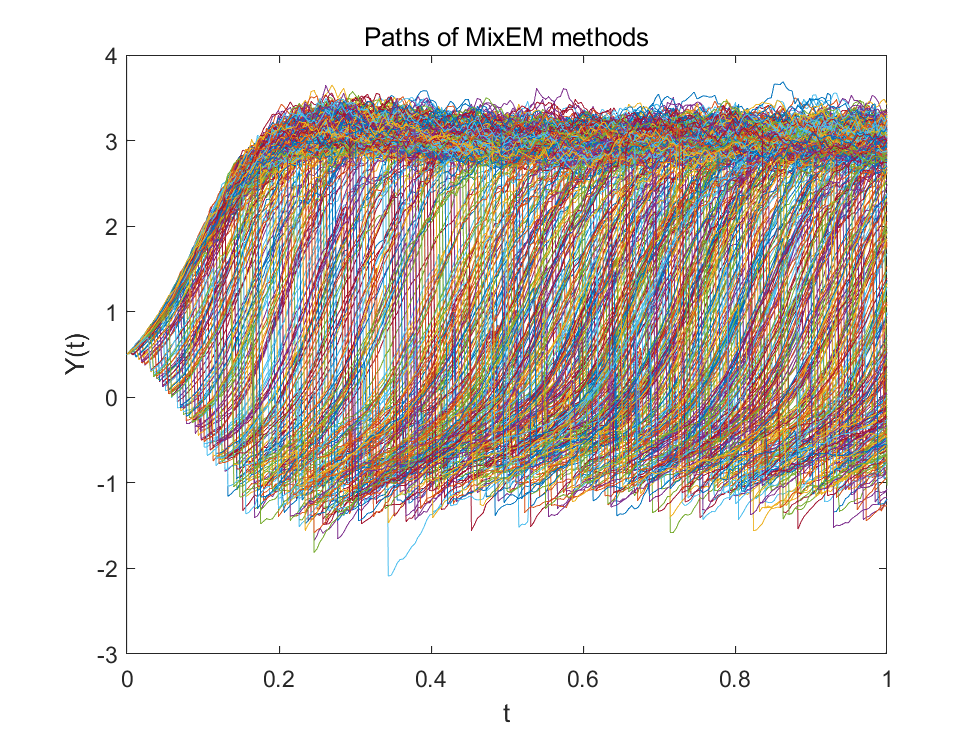} 
\end{minipage}
\caption{Paths of SineEM (left) and MixEM methods (right) for Example \ref{exam:3/2-volatility-model}}
\label{paths-SineEM-MixEm-vola-model}
\end{figure}

It has been verified that Assumptions \ref{ass:inital-val-con}, \ref{ass:Coercivity-condition-Chen}-\ref{ass:Initial-value-is-bounded} are met in this model with $ \gamma =2, \eta =1.5$ and $\bar{p}=1641$ is enough for our setting. A detailed verification is presented in \ref{verfivation-numerical-examples}. In our Monte Carlo approximation, we set the terminal time to $ \mathcal{T}=1 $, use time steps $\Delta t = 2^{-7}, 2^{-8}, 2^{-9}, 2^{-10}, 2^{-11}$, and consider $L =2 \times 10^3$ independent trajectories.

As shown in Figure \ref{fig:double-well}, the  convergence rates of  TanhEM \eqref{ex:tanh-euler-scheme}, TameEM \eqref{ex:tamed-euler-scheme}, SineEM \eqref{ex:sine-euler-scheme} and MixEM \eqref{ex:Mix-euler-scheme} are close to $ \frac{1}{2}$.

\begin{figure}[H]
\centering
\includegraphics[height=4.9cm]{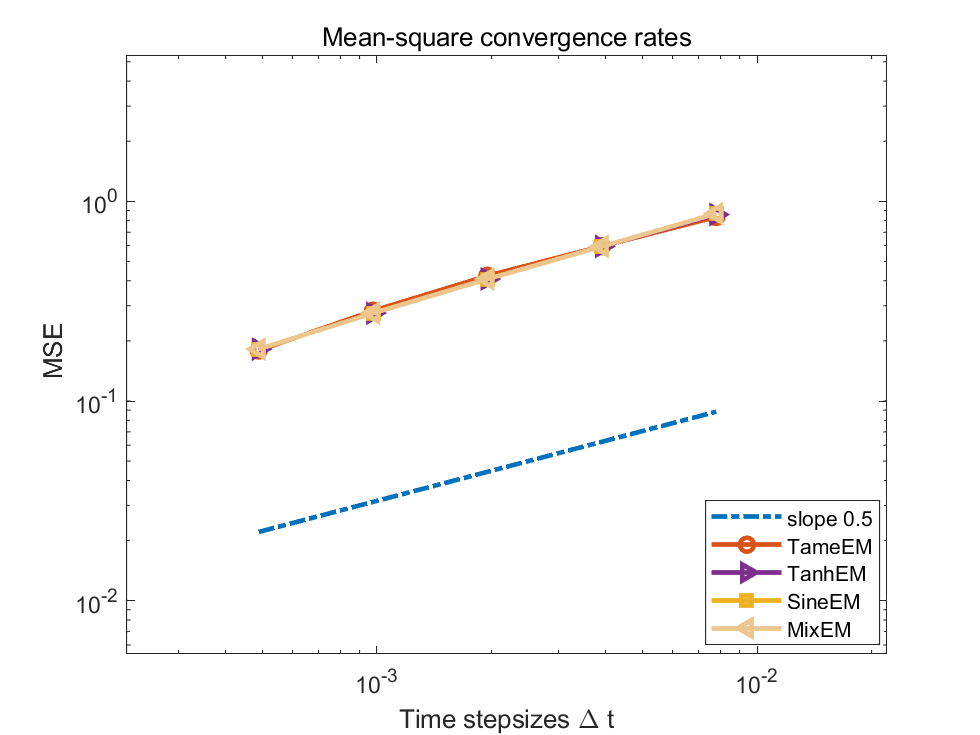} 
\caption{Convergence rates of Euler-type schemes for Example \ref{exam:double-well-model}}
\label{fig:double-well}
\end{figure}

\section*{Acknowledgements}
This work was supported by the NSF grant of China (No.12371417), the Postdoctoral Fellowship Program of CPSF under Grant Number GZC2024205 and 2025M773122, and the innovative project of graduate students of Central South University (2023ZZTS0161).

\vskip6mm
 \bibliographystyle{plain}
\bibliography{ref.bib}


\appendix

\section{Proof of Lemma \ref{lem:moment-bounds-of-numerical-solution}}
\label{appen:proof-moment-bound-modi-method}
\textit{Proof.}
Let $\mathfrak{R}>0$ be a sufficiently large constant and define a sequence of decreasing subevents
    $$
    \Omega^{i}_{\mathfrak{R},k}:=\{\omega\in \Omega :|Y_{t_{j_{1}}}^{i,N}|\le \mathfrak{R}, \, j_{1}=0,\cdots,k\}, 
    $$
with their complements denoted by $\Omega_{\mathfrak{R},k}^{i,c}$.  
    
To begin, we demonstrate that the high-order moment boundedness is preserved within a chosen family of subevents. For integer $\bar{p} \geq 1$, we notice that
\begin{align}  \label{I}
&\mathbb{E}\left[\mathds{1}_{\Omega^{i}_{\mathfrak{R},k+1}}|Y_{t_{k+1}}^{i,N}|^{2\bar{p}}\right]  \notag \\
\le ~& \mathbb{E} \left[\mathds{1}_{\Omega^{i}_{\mathfrak{R},k}} |Y_{t_{k+1}}^{i,N}-Y_{t_{k}}^{i,N}+Y_{t_{k}}^{i,N}|^{2\bar{p}}\right]  \notag \\
    \le ~& \mathbb{E}\left[\mathds{1}_{\Omega^{i}_{\mathfrak{R},k}}|Y_{t_{k}}^{i,N}|^{2\bar{p}}\right] \notag \\
    &~ + \mathbb{E}\left[\mathds{1}_{\Omega^{i}_{\mathfrak{R},k}}|Y_{t_{k}}^{i,N}|^{2\bar{p}-2} \cdot \left(2\bar{p} \langle Y_{t_{k}}^{i,N},Y_{t_{k+1}}^{i,N}-Y_{t_{k}}^{i,N}\rangle + \bar{p}(2\bar{p}-1)|Y_{t_{k+1}}^{i,N}-Y_{t_{k}}^{i,N}|^{2}\right)\right] \notag \\
&~+C \sum_{\kappa=3}^{2\bar{p}} \mathbb{E}\left[\mathds{1}_{\Omega^{i}_{\mathfrak{R},k}}|Y_{t_{k}}^{i,N}|^{2\bar{p}-\kappa}|Y_{t_{k+1}}^{i,N}-Y_{t_{k}}^{i,N}|^{\kappa}\right]  \notag \\
:= ~& \mathbb{E}\left[\mathds{1}_{\Omega^{i}_{\mathfrak{R},k}}|Y_{t_{k}}^{i,N}|^{2\bar{p}}\right] + J_{1} +J_{2}.
\end{align}

According to \eqref{eq:Modified-Euler-method} and the martingale property, we have
\begin{align*} 
    J_{1} 
    = &~ \mathbb{E}  \left[\mathds{1}_{\Omega^{i}_{\mathfrak{R},k}}|Y_{t_{k}}^{i,N}|^{2\bar{p}-2} \cdot 2\bar{p}\cdot \bigg\langle Y_{t_{k}}^{i,N},\Gamma_{1}\left(f\left(t_{k},Y^{i,N}_{t_{k}},\rho_{t_{k}}^{Y,N}\right),\Delta t \right)\Delta t-f\left(t_{k},Y^{i,N}_{t_{k}},\rho_{t_{k}}^{Y,N}\right)\Delta t \bigg\rangle  \right]  \notag \\
    &~ +\mathbb{E} \left[\mathds{1}_{\Omega^{i}_{\mathfrak{R},k}}|Y_{t_{k}}^{i,N}|^{2\bar{p}-2} \cdot 2\bar{p}\cdot \bigg\langle Y_{t_{k}}^{i,N},f\left(t_{k},Y^{i,N}_{t_{k}},\rho_{t_{k}}^{Y,N}\right) \bigg\rangle  \right] \Delta t \notag \\
    &~ + \mathbb{E}
    \left[\mathds{1}_{\Omega^{i}_{\mathfrak{R},k}}|Y_{t_{k}}^{i,N}|^{2\bar{p}-2}\cdot \bar{p}(2\bar{p}-1)\left|\Gamma_{1}\left(f\left(t_{k},Y^{i,N}_{t_{k}},
\rho_{t_{k}}^{Y,N} \right),\Delta t\right)\Delta t\right|^{2}\right]  \notag \\
    &~ + \mathbb{E} 
    \left[\mathds{1}_{\Omega^{i}_{\mathfrak{R},k}}|Y_{t_{k}}^{i,N}|^{2\bar{p}-2}\cdot \bar{p}(2\bar{p}-1)\left|\sum_{j=1}^{m}\Gamma_{2}\left(g^{j}\left(t_{k},Y^{i,N}_{t_{k}},
\rho_{t_{k}}^{Y,N} \right),\Delta t\right)\Delta W_{j}^{i}(t_{k})\right|^{2}\right] \notag \\
&~+\mathbb{E}
\left[\mathds{1}_{\Omega^{i}_{\mathfrak{R},k}}|Y_{t_{k}}^{i,N}|^{2\bar{p}-2}\cdot \bar{p}(2\bar{p}-1)\left|\int_{t_{k}}^{t_{k+1}}\int_{\mathcal{E}}\Gamma_{3}\left(h\left(t_{k},Y^{i,N}_{t_{k}},
\rho_{t_{k}}^{Y,N},v \right),\Delta t \right)\tilde{p}^{i}_{\varphi}(dv,ds)\right|^{2}\right].
\end{align*}
By the Cauchy-Schwarz inequality, Assumption \ref{ass:Coefficient-comparison-conditions-of-f}, Assumption \ref{ass:Gamma-control-conditions} and the polynominal growth of the cofficient $f$ in
\eqref{ineq:growth-condition-of-a}, we get
\begin{align}  \label{ineq:I1}
    J_{1} 
\le &~ \underbrace{C\mathbb{E} \left[\mathds{1}_{\Omega^{i}_{\mathfrak{R},k}}|Y_{t_{k}}^{i,N}|^{2\bar{p}-1}\left(1+\left|Y^{i,N}_{t_{k}}\right|^{\gamma+1}+\mathbb{W}_{2}\left(\rho_{t_{k}}^{Y,N},\delta_{0}\right)\right)^{\hat{\gamma}}\right]\Delta t^{1+\hat{\delta}}}_{J_{1,1}}  \notag  \\
&~+ \underbrace{C\mathbb{E} \left[\mathds{1}_{\Omega^{i}_{\mathfrak{R},k}}|Y_{t_{k}}^{i,N}|^{2\bar{p}-2}\left(1+\left|Y^{i,N}_{t_{k}}\right|^{\gamma+1}+\mathbb{W}_{2}\left(\rho_{t_{k}}^{Y,N},\delta_{0}\right)\right)^{2}\right]\Delta t^{2}}_{J_{1,2}}  \notag \\
&~ +\underbrace{\mathbb{E} \left[\mathds{1}_{\Omega^{i}_{\mathfrak{R},k}}|Y_{t_{k}}^{i,N}|^{2\bar{p}-2} \cdot 2\bar{p}\cdot \bigg\langle Y_{t_{k}}^{i,N},f\left(t_{k},Y^{i,N}_{t_{k}},\rho_{t_{k}}^{Y,N}\right) \bigg\rangle  \right] \Delta t}_{J_{1,3}}\notag \\
&~ + \underbrace{\mathbb{E} 
\left[\mathds{1}_{\Omega^{i}_{\mathfrak{R},k}}|Y_{t_{k}}^{i,N}|^{2\bar{p}-2}\cdot \bar{p}(2\bar{p}-1)\sum_{j=1}^{m}\left|g_{j}\left(t_{k},Y^{i,N}_{t_{k}},
\rho_{t_{k}}^{Y,N} \right)\right|^{2}\right] \Delta t}_{J_{1,4}} \notag \\
&~+\mathbb{E}
\left[\mathds{1}_{\Omega^{i}_{\mathfrak{R},k}}|Y_{t_{k}}^{i,N}|^{2\bar{p}-2}\cdot \bar{p}(2\bar{p}-1)\int_{t_{k}}^{t_{k+1}}\int_{\mathcal{E}}\left|h\left(t_{k},Y^{i,N}_{t_{k}},
\rho_{t_{k}}^{Y,N},v \right)\right|^{2}\varphi(dv)ds\right].
\end{align}
Using the elementary inequality $
(\sum_{i=1}^{k_1}|a_i|)^{l_1} \leq k_1^{l_1-1} \sum_{i=1}^{k_1}|a_i|^{l_1}, \,  \forall \, l_1 >0, \, 
a_i \in \R, \, i=1, \dots, k_1, \, k_1 \in \mathbb{N}
$, Assumption \ref{ass:Gamma-control-conditions}, the polynomial growth for the coefficients in \eqref{ineq:growth-condition-of-a} and \eqref{eq:poly-growth-diffusion}, we can derive
\begin{align}  \label{ineq:I2}
  J_{2} \le &~C \sum_{\kappa=3}^{2\bar{p}} \mathbb{E}\left[\mathds{1}_{\Omega^{i}_{\mathfrak{R},k}}|Y_{t_{k}}^{i,N}|^{2\bar{p}-\kappa}\left|\Gamma_{1}\left(f\left(t_{k}, Y_{t_{k}}^{i,N},\rho_{t_{k}}^{Y,N}\right),\Delta t\right) \Delta t \right|^{\kappa}\right]  \notag \\
  &~+ C \sum_{\kappa=3}^{2\bar{p}} \mathbb{E}\left[\mathds{1}_{\Omega^{i}_{\mathfrak{R},k}}|Y_{t_{k}}^{i,N}|^{2\bar{p}-\kappa}\left|\sum_{j=1}^{m} \Gamma_{2}\left(g_{j}\left(t_{k},Y_{t_{k}}^{i,N},\rho_{t_{k}}^{Y,N}\right),\Delta t \right) \Delta{W_{j}^{i}}(t_{k})\right|^{\kappa}\right] \notag \\
  &~+ C \sum_{\kappa=3}^{2\bar{p}} \mathbb{E}\left[\mathds{1}_{\Omega^{i}_{\mathfrak{R},k}}|Y_{t_{k}}^{i,N}|^{2\bar{p}-\kappa}\left|\int_{t_{k}}^{t_{k+1}}\int_{\mathcal{E}}\Gamma_{3}\left(h\left(t_{k},Y_{t_{k}}^{i,N},\rho_{t_{k}}^{Y,N},v\right),\Delta t \right)\tilde{p}^{i}_{\varphi}(dv,ds)\right|^{\kappa}\right]  \notag \\
  \le &~ \underbrace{C \sum_{\kappa=3}^{2\bar{p}} \mathbb{E} \left[\mathds{1}_{\Omega^{i}_{\mathfrak{R},k}}|Y_{t_{k}}^{i,N}|^{2\bar{p}-\kappa}\left(1+\left|Y^{i,N}_{t_{k}}\right|^{\gamma+1}+\mathbb{W}_{2}\left(\rho_{t_{k}}^{Y,N},\delta_{0}\right)\right)^{\kappa}\right]\Delta t^{\kappa}}_{J_{2,1}}   \notag \\
   &~ + \underbrace{C \sum_{\kappa=3}^{2\bar{p}} \mathbb{E} \left[\mathds{1}_{\Omega^{i}_{\mathfrak{R},k}}|Y_{t_{k}}^{i,N}|^{2\bar{p}-\kappa}\left(1+\left|Y^{i,N}_{t_{k}}\right|^{\frac{\gamma}{2}+1}+\mathbb{W}_{2}\left(\rho_{t_{k}}^{Y,N},\delta_{0}\right)\right)^{\kappa}\right]\Delta t^{\frac{\kappa}{2}}}_{J_{2,2}} \notag \\
   &~+ C \sum_{\kappa=3}^{2\bar{p}} \mathbb{E}
\left[\mathds{1}_{\Omega^{i}_{\mathfrak{R},k}}|Y_{t_{k}}^{i,N}|^{2\bar{p}-\kappa}\int_{t_{k}}^{t_{k+1}}\int_{\mathcal{E}}\left|h\left(t_{k},Y^{i,N}_{t_{k}},
\rho_{t_{k}}^{Y,N},v \right)\right|^{\kappa}\varphi(dv)ds\right].
\end{align}
Coming \eqref{ineq:I1} with \eqref{ineq:I2}, and using the Young inequality and Assumption \ref{ass:Coercivity-condition-Chen}, we show
\begin{align}
    J_{1}+J_{2} \le &~ J_{1,1} + J_{1,2} + J_{2,1} +J_{2,2} + J_{1,3} + J_{1,4} \notag \\
 &~+  \mathbb{E}
\left[\mathds{1}_{\Omega^{i}_{\mathfrak{R},k}} \int_{t_{k}}^{t_{k+1}}(1+(2\bar{p}-2)\theta)\int_{\mathcal{E}}\left|h\left(t_{k},Y^{i,N}_{t_{k}},
\rho_{t_{k}}^{Y,N},v \right)\right|^{2\bar{p}}\varphi(dv)ds\right] \notag \\
&~+C\mathbb{E} \left[\mathds{1}_{\Omega^{i}_{\mathfrak{R},k}} \left|Y_{t_{k}}^{i,N}\right|^{2\bar{p}}\right] \Delta t \notag \\
\le &~ J_{1,1} + J_{1,2} + J_{2,1} + J_{2,2}
\notag \\
 &~+C\mathbb{E} \left[\mathds{1}_{\Omega^{i}_{\mathfrak{R},k}} \left|Y_{t_{k}}^{i,N}\right|^{2\bar{p}}\right] \Delta t +C\mathbb{E} \left[\mathds{1}_{\Omega^{i}_{\mathfrak{R},k}}\left(1+\left|Y_{t_{k}}^{i,N}\right|^{2\bar{p}}+\mathbb{W}_{2}^{2\bar{p}}(\rho_{t_{k}}^{Y,N},\delta_{0})\right)\right] \Delta t.
\end{align}

\noindent Further, by Lemma 4.2 of \cite{kumar2022well}, we have that for any $\hat{p} \geq 2$,
\begin{align} \label{eq:W_2^p}
\mathbb{W}_2^{\hat{p}}\left(\rho_{t_k}^{Y, N}, \delta_0\right)=\mathbb{W}_2^{\hat{p}}\left(\frac{1}{N} \sum_{i=1}^N \delta_{Y_{t_{k}}^{i, N}}, \delta_0\right) \le \frac{1}{N} \sum_{i=1}^N\left|Y_{t_{k}}^{i, N}\right|^{\hat{p}}.
\end{align}
From the above estimations, we thus arrive at
\begin{align}
    &~ \sup_{i\in \mathcal{I}_{N}} \mathbb{E}\left[1_{\Omega^{i}_{\mathfrak{R},k+1}}|Y_{t_{k+1}}^{i,N}|^{2\bar{p}}\right]  \notag \\
    \le &~  C \Delta t +  (1+ C \Delta t) \sup_{i\in \mathcal{I}_{N}} \mathbb{E} \left[\mathds{1}_{\Omega^{i}_{\mathfrak{R},k}}|Y_{t_{k}}^{i,N}|^{2\bar{p}}\right] + C \sup_{i\in \mathcal{I}_{N}} \mathbb{E} \left[\mathds{1}_{\Omega^{i}_{\mathfrak{R},k}}|Y_{t_{k}}^{i,N}|^{2\bar{p}-1+\hat{\gamma}(\gamma+1)}\right] \Delta t^{1+\hat{\delta}}   \notag  \\
    &~+ C \sum_{\kappa=2}^{2\bar{p}}  \sup_{i\in \mathcal{I}_{N}} \mathbb{E} \left[\mathds{1}_{\Omega^{i}_{\mathfrak{R},k}}|Y_{t_{k}}^{i,N}|^{2\bar{p}+\gamma \kappa}\right] \Delta t^{\kappa}     
    + C \sum_{\kappa=3}^{2\bar{p}}  \sup_{i\in \mathcal{I}_{N}} \mathbb{E} \left[\mathds{1}_{\Omega^{i}_{\mathfrak{R},k}}|Y_{t_{k}}^{i,N}|^{2\bar{p}+\frac{\gamma \kappa}{2}}\right] \Delta t^{\frac{\kappa}{2}}.    
\end{align}
Choosing $\mathfrak{R}=\mathfrak{R}(\Delta t) = \Delta t^{-1/G(\gamma,\hat{\gamma},\hat{\delta})}$ with $G(\gamma,\hat{\gamma},\hat{\delta})=\frac{\hat{\gamma}(\gamma+1)-1}{\hat{\delta}}\vee 3\gamma $, we get 
\begin{align*}
 \mathds{1}_{\Omega^{i}_{\mathfrak{R},k}}|Y_{t_{k}}^{i,N}|^{2\bar{p}-1+\hat{\gamma}(\gamma+1)} \Delta t^{1+\hat{\delta}} = \mathds{1}_{\Omega^{i}_{\mathfrak{R},k}}|Y_{t_{k}}^{i,N}|^{2\bar{p}} \Delta t \left(\mathds{1}_{\Omega^{i}_{\mathfrak{R},k}}|Y_{t_{k}}^{i,N}|^{\hat{\gamma}(\gamma+1)-1} \Delta t ^{\hat{\delta}}\right) \le  C\mathds{1}_{\Omega^{i}_{\mathfrak{R},k}}|Y_{t_{k}}^{i,N}|^{2\bar{p}} \Delta t,   \\
 \mathds{1}_{\Omega^{i}_{\mathfrak{R},k}}|Y_{t_{k}}^{i,N}|^{2\bar{p}+\gamma \kappa} \Delta t^{\kappa} = \mathds{1}_{\Omega^{i}_{\mathfrak{R},k}}|Y_{t_{k}}^{i,N}|^{2\bar{p}} \Delta t \left(\mathds{1}_{\Omega^{i}_{\mathfrak{R},k}}|Y_{t_{k}}^{i,N}|^{\gamma \kappa} \Delta t^{\kappa-1} \right) \le C \mathds{1}_{\Omega^{i}_{\mathfrak{R},k}}|Y_{t_{k}}^{i,N}|^{2\bar{p}} \Delta t, \quad \kappa =2,\cdots,2\bar{p},  \\
 \mathds{1}_{\Omega^{i}_{\mathfrak{R},k}}|Y_{t_{k}}^{i,N}|^{2\bar{p}+\frac{\gamma \kappa}{2}} \Delta t^{\frac{\kappa}{2}}= \mathds{1}_{\Omega^{i}_{\mathfrak{R},k}}|Y_{t_{k}}^{i,N}|^{2\bar{p}} \Delta t \left( \mathds{1}_{\Omega^{i}_{\mathfrak{R},k}}|Y_{t_{k}}^{i,N}|^{\frac{\gamma \kappa}{2}} \Delta t^{\frac{\kappa}{2}-1}\right)  \le C \mathds{1}_{\Omega^{i}_{\mathfrak{R},k}}|Y_{t_{k}}^{i,N}|^{2\bar{p}} \Delta t, \quad \kappa=3,\cdots,2\bar{p}.
\end{align*}
From the above estimations, it follows that
\begin{align*}
    \sup_{i\in\mathcal{I}_{N}} \mathbb{E}\left[\mathds{1}_{\Omega^{i}_{\mathfrak{R},k+1}}|Y_{t_{k+1}}^{i,N}|^{2\bar{p}}\right] \le &~  C \Delta t +  (1+ C \Delta t) \sup_{i\in\mathcal{I}_{N}} \mathbb{E} \left[\mathds{1}_{\Omega^{i}_{\mathfrak{R},k}}|Y_{t_{k}}^{i,N}|^{2\bar{p}}\right].
\end{align*}
According to the Gronwall inequality, we have 
\begin{align} \label{ineq:Moments-bounded-of-sets}
    \sup_{i\in\mathcal{I}_{N}} \mathbb{E}\left[\mathds{1}_{\Omega^{i}_{\mathfrak{R},k+1}}|Y_{t_{k+1}}^{i,N}|^{2\bar{p}}\right] \le &~  C \left(1+\mathbb{E}\left[|Y_{0}^{i,N}|^{2\bar{p}}\right]\right).
\end{align}
In the sequel, we focus on the estimate of 
$\mathbb{E}\left[\mathds{1}_{\Omega_{\mathfrak{R},k}^{i,c}}\left|Y_{t_{k}}^{i,N}\right|^{2p}\right]$. We first infer from \eqref{eq:Modified-Euler-method} that
\begin{align*}
 \left|Y_{t_{k+1}}^{i,N}\right| \le &~ \left|Y_{0}^{i,N}\right| + \sum_{j_{1}=0}^{k} \left|\Gamma_{1}\left(f\left(t_{j_{1}}, Y_{t_{j_{1}}}^{i,N},\rho_{t_{j_{1}}}^{Y,N}\right),\Delta t\right) \Delta t \right|  \notag \\
 &~+ \sum_{j_{1}=0}^{k} \left|\sum_{j=1}^{m} \Gamma_{2}\left(g_{j}\left(t_{j_{1}},Y_{t_{j_{1}}}^{i,N},\rho_{t_{j_{1}}}^{Y,N}\right),\Delta t \right) \Delta{W_{j}^{i}}(t_{j_{1}})\right| \notag \\
    &~+ \sum_{j_{1}=0}^{k}\left|\int_{t_{j_{1}}}^{t_{j_{1}+1}}\int_{\mathcal{E}}\Gamma_{3}\left(h\left(t_{j_{1}},Y_{t_{j_{1}}}^{i,N},\rho_{t_{j_{1}}}^{Y,N},v\right),\Delta t \right)\tilde{p}^{i}_{\varphi}(dv,ds)\right|. 
\end{align*}
In view of Assumption \ref{ass:Gamma-control-conditions}, we obtain
\begin{align*} 
\mathbb{E}\left[\left|Y_{t_{k+1}}^{i,N}\right|^{2\bar{p}} \right]
  \le &~C\mathbb{E}\left[\left|Y_{0}^{i,N}\right|^{2\bar{p}}\right] + C\Delta t^{-2\bar{p}\alpha_{1}} + C \Delta t^{-2\bar{p}\left(\alpha_{2}+\frac{1}{2}\right)}+C\Delta t^{-2\bar{p}\left(\alpha_{3}+1\right)}  \notag \\
  \le &~C\mathbb{E}\left[\left|Y_{0}^{i,N}\right|^{2\bar{p}}\right]+C\Delta t^{-2\bar{p}\bar{\alpha}},
\end{align*}
where $\bar{\alpha} = \max\{\alpha_{1},\alpha_{2}+\frac{1}{2},\alpha_{3}+1\}$. 
Therefore, we can conclude that
\begin{align} \label{ineq:high order moment of numerical solution}
  \mathbb{E}\left[\left|Y_{t_{k}}^{i,N}\right|^{2\bar{p}} \right]
  \le ~C\mathbb{E}\left[\left|Y_{0}^{i,N}\right|^{2\bar{p}}\right]+C\Delta t^{-2\bar{p}\bar{\alpha}}.
\end{align}

\noindent Before proceeding further with the estimate of 
$\mathbb{E}\left[\mathds{1}_{\Omega_{\mathfrak{R},k}^{i,c}}\left|Y_{t_{k}}^{i,N}\right|^{2p}\right]$, 
 we notice that,
\begin{align}  \label{Complementary-iterative-formula}
    \mathds{1}_{\Omega_{\mathfrak{R},k}^{i,c}}=&~1-\mathds{1}_{\Omega^{i}_{\mathfrak{R},k}}=1-\mathds{1}_{\Omega^{i}_{\mathfrak{R},k-1}}\mathds{1}_{\big|Y_{t_{k}}^{i,N}\big|\le \mathfrak{R}(\Delta t)} = \mathds{1}_{\Omega_{\mathfrak{R},k-1}^{i,c}}+\mathds{1}_{\Omega^{i}_{\mathfrak{R},k-1}}\mathds{1}_{\left|Y_{t_{k}}^{i,N}\right|> \mathfrak{R}} \notag \\
    =&~\sum_{j_{1}=0}^{k}\mathds{1}_{\Omega^{i}_{\mathfrak{R},j_{1}-1}}\mathds{1}_{\big|Y_{t_{j_{1}}}^{i,N}\big|> \mathfrak{R}},
\end{align}
where $\mathds{1}_{\Omega^{i}_{\mathfrak{R},-1}}=1$. 

Using \eqref{Complementary-iterative-formula}, the $\rm H\ddot{o}lder$ inequality with $\frac{1}{p'}+\frac{1}{q'}=1$ and the Chebyshev inequality, we achieve that for $p \geq 1$,
\begin{align}   \label{ineq:Moment-Bounded-Estimation-of-Set-Complements}
\mathbb{E}\left[\mathds{1}_{\Omega_{\mathfrak{R},k}^{i,c}}\left|Y_{t_{k}}^{i,N}\right|^{2p}\right] 
    = & \sum_{j_{1}=0}^{k} \mathbb{E} \left[\mathds{1}_{\Omega^{i}_{\mathfrak{R},j_{1}-1}}\mathds{1}_{|Y_{t_{j_{1}}}^{i,N}|> \mathfrak{R}} \left|Y_{t_{k}}^{i,N}\right|^{2p}\right] \notag \\
    \le &~\sum_{j_{1}=0}^{k} \left(\mathbb{E}\left[\left|Y_{t_{k}}^{i,N}\right|^{2p\cdot p'}\right]\right)^{\frac{1}{p'}}\left(\mathbb{E}\left[\mathds{1}_{\Omega^{i}_{\mathfrak{R},j_{1}-1}}\mathds{1}_{|Y_{t_{j_{1}}}^{i,N}|> \mathfrak{R}}\right]\right)^{\frac{1}{q'}}  \notag \\
    =&~\left(\mathbb{E}\left[\left|Y_{t_{k}}^{i,N}\right|^{2p\cdot p'}\right]\right)^{\frac{1}{p'}} \sum_{j_{1}=0}^{k} \left(\mathbb{P}\left[\mathds{1}_{\Omega^{i}_{\mathfrak{R},j_{1}-1}}|Y_{t_{j_{1}}}^{i,N}|> \mathfrak{R}\right]\right)^{\frac{1}{q'}} \notag \\
    \le &~\left(\mathbb{E}\left[\left|Y_{t_{k}}^{i,N}\right|^{2p\cdot p'}\right]\right)^{\frac{1}{p'}} \sum_{j_{1}=0}^{k} \frac{\left(\mathbb{E}\left[\mathds{1}_{\Omega^{i}_{\mathfrak{R},j_{1}-1}}\left|Y_{t_{j_{1}}}^{i,N}\right|^{2\bar{p}}\right]\right)^{\frac{1}{q'}}}{\left(\mathfrak{R}\right)^{2\bar{p}/q'}},
\end{align}
where $q'=\frac{2\bar{p}}{(2\bar{\alpha} p+1)G(\gamma,\hat{\gamma},\hat{\delta})} >1$, as $p \leq \frac{2\bar{p}-G(\gamma,\hat{\gamma},\hat{\delta})}{2+2\bar{\alpha} G(\gamma,\hat{\gamma},\hat{\delta})}$.

Since $p \leq \frac{2\bar{p}-G(\gamma,\hat{\gamma},\hat{\delta})}{2+2\bar{\alpha} G(\gamma,\hat{\gamma},\hat{\delta})}$, it follows that $pp'\le \bar{p}$. 
Applying the $\rm H\ddot{o}lder$ inequality and \eqref{ineq:high order moment of numerical solution}, we obtain
\begin{align} \label{ineq:PP'-th moment of Y_{tn}}
\left(\mathbb{E}\left[\left|Y_{t_{k}}^{i,N}\right|^{2p\cdot p'}\right]\right)^{\frac{1}{p'}} \le &~\left(\mathbb{E}\left[\left|Y_{t_{k}}^{i,N}\right|^{2p\cdot p'\cdot \frac{\bar{p}}{pp'}}\right]\right)^{\frac{pp'}{\bar{p}}\frac{1}{p'}}\notag \\
  \le &~C\left(1+\mathbb{E}\left[\left|Y_{0}^{i,N}\right|^{2\bar{p}}\right]\right)^{\frac{p}{\bar{p}}} + C\Delta t^{-2\bar{\alpha} p}.
\end{align}
Substituting \eqref{ineq:PP'-th moment of Y_{tn}} and \eqref{ineq:Moments-bounded-of-sets} into \eqref{ineq:Moment-Bounded-Estimation-of-Set-Complements}, with $\mathfrak{R}(\Delta t) = \Delta t^{-1/G(\gamma,\hat{\gamma},\hat{\delta})}$ leads to
\begin{align}  \label{ineq:Bounded-moment-estimates-of-complements}
    &\mathbb{E}\left[\mathds{1}_{\Omega_{\mathfrak{R},k}^{i,c}}\left|Y_{t_{k}}^{i,N}\right|^{2p}\right] \notag \\
    \le &~C(k+1)\Delta t^{2\bar{\alpha} p+1}\left(C\Delta t^{-2\bar{\alpha} p}+C\left(1+\mathbb{E}\left[\left|Y_{0}^{i,N}\right|^{2\bar{p}}\right]\right)^{\frac{p}{\bar{p}}}\right)\cdot \left[C\left(1+\mathbb{E}\left[\left|Y_{0}^{i,N}\right|^{2\bar{p}}\right]\right)\right]^{\frac{1}{q'}} \notag \\
    \le&~ C\left(1+\mathbb{E}\left[\left|Y_{0}^{i,N}\right|^{2\bar{p}}\right]\right)^{\frac{p}{\bar{p}}+\frac{1}{q'}}.
\end{align}
Combining the $\rm H\ddot{o}lder$ inequality, \eqref{ineq:Moments-bounded-of-sets} and \eqref{ineq:Bounded-moment-estimates-of-complements}, we show
\begin{align*}
   \sup_{i\in\mathcal{I}_{N}}\mathbb{E}\left[|Y_{t_{k}}^{i,N}|^{2p}\right] 
    =&~\sup_{i\in\mathcal{I}_{N}}\mathbb{E}\left[\mathds{1}_{\Omega_{\mathfrak{R},k}}\left|Y_{t_{k}}^{i,N}\right|^{2p}\right]+\sup_{i\in\mathcal{I}_{N}}\mathbb{E}\left[\mathds{1}_{\Omega_{\mathfrak{R},k}^{c}}\left|Y_{t_{k}}^{i,N}\right|^{2p}\right]  \\
    \le &~\left(\sup_{i\in\mathcal{I}_{N}}\mathbb{E}\left[\mathds{1}_{\Omega_{\mathfrak{R},k}}\left|Y_{t_{k}}^{i,N}\right|^{2\bar{p}}\right]\right)^{\frac{p}{\bar{p}}}+\sup_{i\in\mathcal{I}_{N}}\mathbb{E}\left[\mathds{1}_{\Omega_{\mathfrak{R},k}^{c}}\left|Y_{t_{k}}^{i,N}\right|^{2p}\right]   \\
    \le &~C\left(1+\mathbb{E}\left[\left|Y_{0}^{i,N}\right|^{2\bar{p}}\right]\right)^{\frac{p}{\bar{p}}} + C\left(1+\mathbb{E}\left[\left|Y_{0}^{i,N}\right|^{2\bar{p}}\right]\right)^{\frac{p}{\bar{p}}+\frac{1}{q'}} \\
    \le &~C\left(1+\left(\mathbb{E}\left[|Y_{0}^{i,N}|^{2\bar{p} }\right]\right)^{\beta}\right).
\end{align*}
where $\beta >0 $. Then, by the $\rm H\ddot{o}lder$ inequality, \eqref{Moments-numer-bound-stro-con} is shown to hold for non-integer values of $p \geq 1$, thereby completing the proof.
\qed

\section{Proof of Theorem \ref{thm:con-result-particle-scheme}}
\label{sec:proof}
\renewcommand{\thelemma}{B.\arabic{lemma}}  
\setcounter{lemma}{0} 
Now, we prove the convergence rate for the Euler-type schemes \eqref{eq:continuous-Modified-Euler-method} in Theorem \ref{thm:con-result-particle-scheme}. The following lemma reveals the strong error estimates between $Y^{i,N}(t)$ and $Y^{i,N}\left(\tau_{n}(t)\right)$.

\begin{lemma}
\label{lem:estimate-of-difference-of-numerical-solution}
Under the same assumptions as Lemma \ref{lem:moment-bounds-of-numerical-solution}, the following estimate holds for the Euler-type schemes \eqref{eq:continuous-Modified-Euler-method}
\begin{align*}
     \sup_{i\in\mathcal{I}_{N}} \sup_{t\in[0,\mathcal{T}]} \mathbb{E} \left[\left|Y^{i,N}(t)-Y^{i,N}\left(\tau_{n}(t)\right)\right|^{2p}\right]  \le C \Delta t \left(1+\left(\mathbb{E}\left[\left|Y_{0}^{i,N}\right|^{2\bar{p}}\right]\right)^{\beta}\right), \quad p\in \big[1,\frac{2\bar{p}-G}{\big(2+2\bar{\alpha} G\big)\big(\gamma+1\big)}\big],
\end{align*}
where $\beta$, $G$ and $\bar{\alpha}$ are from Lemma \ref{lem:moment-bounds-of-numerical-solution} and $\bar{p}$ comes from Assumption \ref{ass:Coercivity-condition-Chen} satisfying $\bar{p}\ge \gamma+1 +G(\bar{\alpha}\gamma+\bar{\alpha}+\frac{1}{2}) $. Additionally, $\gamma$ comes from Assumption \ref{ass:poly-growth-coeff-a}.
\end{lemma}

\begin{proof}
By applying the elementary inequality, the $\rm H\ddot{o}lder $ inequality, and the Burkholder-Davis-Gundy inequality, we obtain
\begin{align*}
    &\mathbb{E} \left[\left|Y^{i,N}(t)-Y^{i,N}\left(\tau_{n}(t)\right)\right|^{2p}\right]  \\
    \le &~ C\big(t-\tau_{n}(t)\big)^{2p-1} \mathbb{E} \left[\int_{\tau_{n}(t)}^{t}\left|\Gamma_{1}\left(f\left(\tau_{n}(s), Y^{i,N}\left(\tau_{n}(s)\right),\rho^{Y,N}_{\tau_{n}(s)}\right),\Delta t\right) \right|^{2p}ds\right]  \notag \\
    &~+ C\big(t-\tau_{n}(t)\big)^{p-1} \mathbb{E} \left[\sum_{j=1}^{m}\int_{\tau_{n}(t)}^{t}\left|\Gamma_{2}\left(g_{j}\left(\tau_{n}(s), Y^{i,N}\left(\tau_{n}(s)\right),\rho^{Y,N}_{\tau_{n}(s)}\right),\Delta t\right) \right|^{2p}ds\right] \notag \\
    &~+ C \mathbb{E} \left[\int_{\tau_{n}(t)}^{t}\int_{\mathcal{E}}\left|\Gamma_{3}\left(h\left(\tau_{n}(s), Y^{i,N}\left(\tau_{n}(s)\right),\rho^{Y,N}_{\tau_{n}(s)},v\right),\Delta t\right) \right|^{2p}\varphi(dv) ds \right].
\end{align*}
One can use Assumption \ref{ass:Gamma-control-conditions} and the polynomial growth condition for the coefficients $f$, $g$ and $h$ in \eqref{ineq:growth-condition-of-a}, \eqref{eq:poly-growth-diffusion} and \eqref{ineq:growth-condition-of-c} to acquire
\begin{align*}
    &\mathbb{E} \left[\left|Y^{i,N}\left(t\right)-Y^{i,N}\left(\tau_{n}(t)\right)\right|^{2p}\right] \notag \\ 
    \le &~ C\Delta t^{2p-1} \mathbb{E} \left[\int_{\tau_{n}(t)}^{t}\left(1+\left|Y^{i,N}\left(\tau_{n}(s)\right)\right|^{2p(\gamma+1)}+\mathbb{W}_{2}^{2p}\left(\rho^{Y,N}_{\tau_{n}(s)},\delta_{0}\right)\right) ds\right]  \notag \\
    &~+ C \Delta t^{p-1} \mathbb{E} \left[\int_{\tau_{n}(t)}^{t}\left(1+\left|Y^{i,N}\left(\tau_{n}(s)\right)\right|^{2p(\frac{\gamma}{2}+1)}+\mathbb{W}_{2}^{2p}\left(\rho^{Y,N}_{\tau_{n}(s)},\delta_{0}\right)\right)ds\right] \notag \\
    &~+ C \mathbb{E} \left[\int_{\tau_{n}(t)}^{t}\left(1+\left|Y^{i,N}\left(\tau_{n}(s)\right)\right|^{2p+\gamma}+\mathbb{W}_{2}^{2p}\left(\rho^{Y,N}_{\tau_{n}(s)},\delta_{0}\right)\right)ds\right].
\end{align*}
In view of \eqref{eq:W_2^p}, one can get
\begin{align}
\label{eq:time-inter-num-2p-est}
    \mathbb{E} \left[\left|Y^{i,N}\left(t\right)-Y^{i,N}\left(\tau_{n}(t)\right)\right|^{2p}\right] 
    \le &~ C\Delta t^{2p} \left(1+\sup_{s\in[\tau_{n}(t),t]} \sup_{i\in\mathcal{I}_{N}}\mathbb{E}\left[\left|Y^{i,N}\left(\tau_{n}(s)\right)\right|^{2p(\gamma+1)}\right]\right)  \notag \\
    &~+ C \Delta t^{p} \left(1+\sup_{s\in[\tau_{n}(t),t]} \sup_{i\in\mathcal{I}_{N}}\mathbb{E}\left[\left|Y^{i,N}\left(\tau_{n}(s)\right)\right|^{2p(\frac{\gamma}{2}+1)}\right]\right) \notag \\
    &~+ C \Delta t \left(1+\sup_{s\in[\tau_{n}(t),t]} \sup_{i\in\mathcal{I}_{N}}\mathbb{E}\left[\left|Y^{i,N}\left(\tau_{n}(s)\right)\right|^{2p+\gamma}\right]\right).
\end{align}
In light of Lemma \ref{lem:moment-bounds-of-numerical-solution}, for all $p\in \big[1,\frac{2\bar{p}-G}{\big(2+2\bar{\alpha} G\big)\big(\gamma+1\big)}\big]$, it follows that
\begin{align*}
  \sup_{i\in\mathcal{I}_{N}} \sup_{t\in[0,\mathcal{T}]} \mathbb{E} \left[\left|Y^{i,N}(t)-Y^{i,N}\left(\tau_{n}(t)\right)\right|^{2p}\right] \le C \Delta t \left(1+\left(\mathbb{E}\left[\left|Y_{0}^{i,N}\right|^{2\bar{p}}\right]\right)^{\beta}\right) .
\end{align*}
\end{proof}

The following lemma states that the processes 
$\{Y^{i,N}(t)\}_{t \in [0,\mathcal{T}]}$, generated by \eqref{eq:continuous-Modified-Euler-method}, possess bounded moments.

\begin{lemma} \label{lem:moment-bounds-continuous-numerical-solution}
Under the same conditions as Lemma \ref{lem:estimate-of-difference-of-numerical-solution}, there exist $C$, $\beta>0$ such that the Euler-type schemes \eqref{eq:continuous-Modified-Euler-method} satisfies
    \begin{align*}
        \sup_{i\in\mathcal{I}_{N}}\sup_{t\in[0,\mathcal{T}]} \mathbb{E} \left[\left|Y^{i,N}(t)\right|^{2p}\right] \le C \left(1+\left(\mathbb{E}\left[\left|Y_{0}^{i,N}\right|^{2\bar{p}}\right]\right)^{\beta}\right), \quad p\in \left[1,\frac{2\bar{p}-G}{\big(2+2\bar{\alpha} G\big)\big(\gamma+1\big)}\right],
    \end{align*}
where $G$ and $\bar{\alpha}$ come from Lemma \ref{lem:moment-bounds-of-numerical-solution} and $\bar{p}$ is from Assumption \ref{ass:Coercivity-condition-Chen} satisfying $\bar{p}\ge \gamma+1 +G(\bar{\alpha}\gamma+\bar{\alpha}+\frac{1}{2})$.
\end{lemma}
\begin{proof}
By utilizing Lemma \ref{lem:moment-bounds-of-numerical-solution} and Lemma \ref{lem:estimate-of-difference-of-numerical-solution}, for all $p\in \big[1,\frac{2\bar{p}-G}{\big(2+2\bar{\alpha} G\big)\big(\gamma+1\big)}\big]$, we reach the desired conclusion that 
\begin{equation*}
\begin{aligned}
 \sup_{i\in\mathcal{I}_{N}} \sup_{t \in [0, \mathcal{T}]} \mathbb{E} \left[\left|Y^{i, N}(t) \right|^{2 p} \right] 
& \leq  C \sup_{i\in\mathcal{I}_{N}} \sup_{t \in [0, \mathcal{T}]}  \mathbb{E} \left[\left|Y^{i, N}(t) - Y^{i, N}\left(\tau_{n}(t)\right) \right|^{2 p}\right] \\
& \quad + C \sup_{i\in\mathcal{I}_{N}} \sup_{t \in [0, \mathcal{T}]}  \mathbb{E} \left[\left|Y^{i, N}\left(\tau_{n}(t)\right) \right|^{2 p} \right] \\
& \leq  
C \left(1+\left(\mathbb{E}\left[\left|Y_{0}^{i,N}\right|^{2\bar{p}}\right]\right)^{\beta}\right).
\end{aligned}
\end{equation*}
\end{proof}
As a direct consequence of
Lemmas \ref{lem:moment-bounds-of-numerical-solution}, \ref{lem:estimate-of-difference-of-numerical-solution} and \ref{lem:moment-bounds-continuous-numerical-solution}, one can get the following useful lemma on error estimates between $f,g,h$ and maps $\Gamma_l, l=1,2,3.$

\begin{lemma} \label{lem:Difference-of-coefficients}
Let Assumptions \ref{ass:inital-val-con}, \ref{ass:Coercivity-condition-Chen},  \ref{ass:poly-growth-coeff-a}-\ref{ass:holder-continuous}, \ref{ass:Gamma-control-conditions}, and Assumption \ref{ass:Coefficient-comparison-conditions-of-Gamma1-Gamma3} hold, there exists $C>0$ such that, for any $t \in [0,\mathcal{T}]$ and sufficient small $\varepsilon>0$,
\begin{align*}
\mathbb{E}\left[\left|f\left(t, Y^{i, N}(t), \rho^{Y, N}_{t}\right)-\Gamma_1\left(f\left(\tau_n(t), Y^{i, N}\left(\tau_{n}(t)\right), \rho^{Y, N}_{\tau_{n}(t)}\right), \Delta t \right)\right|^{2}\right]~
\leq C \Delta t^{\frac{2}{2+\epsilon}}~ \phi(Y_{0}^{i,N}),    
\end{align*} 
\begin{align*}
\sup_{j \in \{1,2,\cdots m\}} \mathbb{E}\left[\left|g_{j}\left(t, Y^{i, N}(t), \rho^{Y, N}_{t}\right)-\Gamma_2\left(g_{j}\left(\tau_n(t), Y^{i, N}\left(\tau_{n}(t)\right), \rho^{Y, N}_{\tau_{n}(t)}\right), \Delta t \right)\right|^{2}\right]
\leq C \Delta t^{\frac{2}{2+\epsilon}} ~\phi(Y_{0}^{i,N}),
\end{align*}
and
\begin{align*}
 \mathbb{E}\left[\int_{\mathcal{E}}\left|h\left(t, Y^{i, N}(t), \rho^{Y, N}_{t},v\right)-\Gamma_3\left(h\left(\tau_n(t), Y^{i, N}\left(\tau_n(t)\right), \rho^{Y, N}_{\tau_n(t)},v\right), \Delta t \right)\right|^{2} \varphi(dv) \right] 
 \leq C \Delta t^{\frac{2}{2+\epsilon}} ~\phi(Y_{0}^{i,N}),   
\end{align*}
where $\phi(Y_{0}^{i,N}):=1+\left(\mathbb{E}\left[\left|Y_{0}^{i,N}\right|^{2\bar{p}}\right]\right)^{\beta}$, $\beta >0$ comes from Lemma \ref{lem:moment-bounds-of-numerical-solution} and $\bar{p}$ is from Assumption \ref{ass:Coercivity-condition-Chen} satisfying
\begin{equation*}
\bar p \geq \left(\frac{\gamma (2+\varepsilon)}{\varepsilon} \vee \gamma_{1}(\gamma+1) \vee  \gamma_{2}\left(\frac{\gamma}{2}+1\right) \vee \left(\gamma_{3}+\frac{\gamma}{2}\right)\right)(1+\bar{\alpha}G)(\gamma+1)+\frac{G}{2}.
\end{equation*}
\end{lemma}

\begin{proof}
By making use of Assumptions \ref{ass:poly-growth-coeff-a}, \ref{ass:holder-continuous}, \ref{ass:Coefficient-comparison-conditions-of-Gamma1-Gamma3}, the H\"older inequality and \eqref{ineq:growth-condition-of-a}, we have that for sufficient small $\varepsilon>0$,
   \begin{align}
   \label{eq:diff-drift-mean-squ-err}
        &\mathbb{E}\left[\left|f\left(t, Y^{i, N}(t), \rho^{Y, N}_{t}\right)-\Gamma_1\left(f\left(\tau_n(t), Y^{i, N}\left(\tau_n(t)\right), \rho^{Y, N}_{\tau_n(t)}\right), \Delta t \right)\right|^{2}\right] \notag \\
        \le &~C\mathbb{E} \left[\left|f\left(t, Y^{i, N}(t), \rho^{Y, N}_{t}\right)-f\left(t, Y^{i, N}\left(\tau_n(t)\right), \rho^{Y, N}_{\tau_n(t)}\right)\right|^{2}\right] \notag \\
         &~+C\mathbb{E} \left[\left|f\left(t, Y^{i, N}\left(\tau_n(t)\right), \rho^{Y, N}_{\tau_n(t)}\right)-f\left(\tau_n(t), Y^{i, N}\left(\tau_n(t)\right), \rho^{Y, N}_{\tau_n(t)}\right)\right|^{2}\right]  \notag \\
         &~+C\mathbb{E} \left[\left|f\left(\tau_n(t), Y^{i, N}\left(\tau_n(t)\right), \rho^{Y, N}_{\tau_n(t)}\right)-\Gamma_1\left(f\left(\tau_n(t), Y^{i, N}\left(\tau_n(t)\right), \rho^{Y, N}_{\tau_n(t)}\right), \Delta t \right)\right|^{2}\right]  \notag \\
         \le &~C\mathbb{E} \left[\left(1+\left|Y^{i,N}(t)\right|^{\gamma}+\left|Y^{i,N}\left(\tau_n(t)\right)\right|^{\gamma}\right)^{2}\left|Y^{i,N}(t)-Y^{i,N}\left(\tau_n(t)\right)\right|^{2}+\mathbb{W}_{2}^{2}\left(\rho^{Y,N}_{t},\rho^{Y,N}_{\tau_n(t)}\right)\right] \notag \\
         &~+C\left(t-\tau_{n}(t)\right) + C\mathbb{E} \left[\Delta t ^{2\delta_{1}} \left|f \left(\tau_{n}(t),Y^{i,N}\left(\tau_n(t)\right),\rho^{Y,N}_{\tau_n(t)}\right)\right|^{2\gamma_{1}}\right] \notag \\
         \le &~C \left[\mathbb{E}\left(1+\left|Y^{i,N}(t)\right|^{2\gamma\cdot \frac{2+\varepsilon}{\varepsilon}}+\left|Y^{i,N}\left(\tau_n(t)\right)\right|^{2\gamma\cdot \frac{2+\varepsilon}{\varepsilon}}\right)\right]^{\frac{\varepsilon}{2+\varepsilon}}\left(\mathbb{E}\left|Y^{i,N}(t)-Y^{i,N}\left(\tau_n(t)\right)\right|^{2+\varepsilon}\right)^{\frac{2}{2+\varepsilon}} \notag \\
&~+C\mathbb{E}\left[\mathbb{W}_{2}^{2}\left(\rho^{Y,N}_{t},\rho^{Y,N}_{\tau_n(t)}\right)\right]+ C\Delta t + C \Delta t ^{2\delta_{1}} \mathbb{E} \left[ 1+\left|Y^{i,N}\left(\tau_n(t)\right)\right|^{2\gamma_{1}(\gamma+1)}+\mathbb{W}_{2}^{2\gamma_{1}}\left(\rho^{Y,N}_{\tau_n(t)},\delta_{0}\right)\right].
    \end{align}
Since 
$$
\mathbb{W}_2^2\left(\rho^{Y, N}_{t}, \rho^{Y, N}_{\tau_n(t)}\right) \leq \frac{1}{N} \sum_{i=1}^N\left|Y^{i, N}(t)-Y^{i, N}\left(\tau_n(t)\right)\right|^2 \, \hbox{and} \, \, \delta_1 \geq \frac12,
$$
one can utilize Lemma \ref{lem:estimate-of-difference-of-numerical-solution}, \eqref{eq:W_2^p} and Lemma \ref{lem:moment-bounds-continuous-numerical-solution} to show
\begin{align*}
        & \mathbb{E}\left[\left|f\left(t, Y^{i, N}(t), \rho^{Y, N}_{t}\right)-\Gamma_1\left(f\left(\tau_n(t), Y^{i, N}\left(\tau_n(t)\right), \rho^{Y, N}_{\tau_n(t)}\right), \Delta t \right)\right|^{2}\right]  \\
        \le &~C \left[\left(1+\left(\mathbb{E}\left[\left|Y_{0}^{i,N}\right|^{2\bar{p}}\right]\right)^{\beta}\right)\right]^{\frac{\varepsilon}{2+\varepsilon}}\Delta t^{\frac{2}{2+\varepsilon}} +C\Delta t + C \Delta t ^{2\delta_{1}} \left[ 1+\left(\mathbb{E} \left[\left|Y_{0}^{i,N}\right|^{2\bar{p}}\right]\right)^{\beta}\right]\\
         \le &~C \Delta t^{\frac{2}{2+\varepsilon}} \left(1+\left(\mathbb{E}\left[\left|Y_{0}^{i,N}\right|^{2\bar{p}}\right]\right)^{\beta}\right) .
\end{align*}
The similar arguments, together with\eqref{ineq:Polynomial-Growth-of-b}, \eqref{eq:poly-growth-diffusion}, Lemma \ref{lem:estimate-of-difference-of-numerical-solution}, \eqref{eq:W_2^p} and Lemma \ref{lem:moment-bounds-continuous-numerical-solution} yield that for sufficient small $\varepsilon>0$ 
 \begin{align*}
        &\sup_{j \in \{1,2,\cdots m\}} \mathbb{E}\left[\left|g_{j}\left(t, Y^{i, N}(t), \rho^{Y, N}_{t}\right)-\Gamma_2\left(g_{j}\left(\tau_n(t), Y^{i, N}\left(\tau_n(t)\right), \rho^{Y, N}_{\tau_n(t)}\right), \Delta t \right)\right|^{2}\right]
       \\
         \le &~C\mathbb{E} \left[\left(1+\left|Y^{i,N}(t)\right|^{\gamma}+\left|Y^{i,N}\left(\tau_n(t)\right)\right|^{\gamma}\right) \left|Y^{i,N}(t)-Y^{i,N}\left(\tau_n(t)\right)\right|^{2}+\mathbb{W}_{2}^{2}\left(\rho^{Y,N}_{t},\rho^{Y,N}_{\tau_n(t)}\right)\right] \\
         &~+C\left(t-\tau_{n}(t)\right) + C\mathbb{E} \left[\Delta t ^{2\delta_{2}} \left|g_{j} \left(\tau_{n}(t),Y^{i,N}\left(\tau_n(t)\right),\rho^{Y,N}_{\tau_n(t)}\right)\right|^{2\gamma_{2}}\right] \\
         \le &~C \Delta t^{\frac{2}{2+\epsilon}} \left(1+\left(\mathbb{E}\left[\left|Y_{0}^{i,N}\right|^{2\bar{p}}\right]\right)^{\beta}\right),
    \end{align*}
due to $\delta_2 \geq \frac12$. The proof is concluded by performing a similar calculation for $h$ with $\delta_{3}\ge \frac{1}{2}$.
\end{proof}

With the preceding lemmas in hand, we now prove Theorem \ref{thm:con-result-particle-scheme}.

\noindent
\textit{Proof of Theorem \ref{thm:con-result-particle-scheme}.} 
According to \eqref{eq:interat-parti-system} and \eqref{eq:continuous-Modified-Euler-method}, one can use the $\rm It\hat{o}$ formula to acquire
\begin{align*}
    &\left|X^{i,N}(t)-Y^{i,N}(t)\right|^{2}   \\
    =& \int_{0}^{t}2 \bigg\langle X^{i,N}(s)-Y^{i,N}(s), f\left(s,X^{i,N}(s),\rho^{X,N}_{s}\right)-\Gamma_{1}\left(f\left(\tau_{n}(s),Y^{i,N}\left(\tau_n(s)\right),\rho^{Y,N}_{\tau_n(s)}\right),\Delta  t\right) \bigg\rangle  ds  \\
    &~+\sum_{j=1}^{m} \int_{0}^{t} 2 \bigg\langle X^{i,N}(s)-Y^{i,N}(s), g_{j}\left(s,X^{i,N}(s),\rho^{X,N}_{s}\right) -\Gamma_{2}\left(g_{j}\left(\tau_{n}(s),Y^{i,N}\left(\tau_n(s)\right),\rho^{Y,N}_{\tau_n(s)}\right),\Delta  t\right) \bigg\rangle dW_{j}^{i}(s)   \\
    &~+\int_{0}^{t} \int_{\mathcal{E}}2 \bigg\langle X^{i,N}(s)-Y^{i,N}(s), h\left(s,X^{i,N}(s),\rho^{X,N}_{s},v\right)\\
    &\qquad \qquad \qquad \qquad \qquad \qquad \qquad \qquad-\Gamma_{3}\left(h\left(\tau_{n}(s),Y^{i,N}\left(\tau_n(s)\right),\rho^{Y,N}_{\tau_n(s)},v\right),\Delta  t\right) \bigg\rangle \tilde{p}_{\varphi}^{i}(dv,ds) \\
    &~+\int_{0}^{t}\sum_{j=1}^{m}\left|g_{j}\left(s,X^{i,N}(s),\rho^{X,N}_{s}\right)-\Gamma_{2}\left(g_{j}\left(\tau_{n}(s),Y^{i,N}\left(\tau_n(s)\right),\rho^{Y,N}_{\tau_n(s)}\right),\Delta  t\right)\right|^{2} ds  \\
    &~+\int_{0}^{t}\int_{\mathcal{E}} \bigg[\left|X^{i,N}(s)-Y^{i,N}(s)+h\left(s,X^{i,N}(s),\rho^{X,N}_{s},v\right)-\Gamma_{3}\left(h\left(\tau_{n}(s),Y^{i,N}\left(\tau_n(s)\right),\rho^{Y,N}_{\tau_n(s)},v\right),\Delta  t\right)\right|^{2}   \\
    &~\qquad \qquad \quad-\left|X^{i,N}(s)-Y^{i,N}(s)\right|^{2}   -2\bigg\langle X^{i,N}(s)-Y^{i,N}(s), h\left(s,X^{i,N}(s),\rho^{X,N}_{s},v\right)\\
     &~\qquad \qquad \qquad \qquad \qquad \qquad \qquad \qquad \qquad    
     -\Gamma_{3}\left(h\left(\tau_{n}(s),Y^{i,N}\left(\tau_n(s)\right),\rho^{Y,N}_{\tau_n(s)},v\right),\Delta  t\right) \bigg\rangle \bigg] p_{\varphi}^{i}(dv,ds).
\end{align*}
Taking expectation, for $\eta >1$, and then using the Young inequality lead to
\begin{align*}
    &\mathbb{E}\left[\left|X^{i,N}(t)-Y^{i,N}(t)\right|^{2} \right]  \\
    =&~ 2\mathbb{E} \left[\int_{0}^{t} \bigg\langle X^{i,N}(s)-Y^{i,N}(s), f\left(s,X^{i,N}(s),\rho^{X,N}_{s}\right)-\Gamma_{1}\left(f\left(\tau_{n}(s),Y^{i,N}\left(\tau_{n}(s)\right),\rho^{Y,N}_{\tau_{n}(s)}\right),\Delta  t\right) \bigg\rangle ds \right] \\
    &~+\sum_{j=1}^{m}\mathbb{E} \left[\int_{0}^{t}\left|g_{j}\left(s,X^{i,N}(s),\rho^{X,N}_{s}\right)-\Gamma_{2}\left(g_{j}\left(\tau_{n}(s),Y^{i,N}\left(\tau_{n}(s)\right),\rho^{Y,N}_{\tau_{n}(s)} \right),\Delta  t\right)\right|^{2} ds \right]  \\
    &~+\mathbb{E}\left[\int_{0}^{t}\int_{\mathcal{E}} \left|h\left(s,X^{i,N}(s),\rho^{X,N}_{s},v\right)-\Gamma_{3}\left(h\left(\tau_{n}(s),Y^{i,N}\left(\tau_{n}(s)\right),\rho^{Y,N}_{\tau_{n}(s)},v\right),\Delta  t\right)\right|^{2} \varphi(dv)ds \right] \\
    \le &~ 2\mathbb{E} \int_{0}^{t} \bigg\langle X^{i,N}(s)-Y^{i,N}(s), f\left(s,X^{i,N}(s),\rho^{X,N}_{s}\right)-f\left(s,Y^{i,N}(s),\rho^{Y,N}_{s}\right) \bigg\rangle ds  \\
    &~+\eta \sum_{j=1}^{m} \mathbb{E} \left[\int_{0}^{t}\left|g_{j}\left(s,X^{i,N}(s),\rho^{X,N}_{s}\right)-g_{j}\left(s,Y^{i,N}(s),\rho^{Y,N}_{s}\right)\right|^{2} ds \right]  \\
    &~+ \eta \mathbb{E} \left[\int_{0}^{t}\int_{\mathcal{E}} \left|h\left(s,X^{i,N}(s),\rho^{X,N}_{s},v\right)-h\left(s,Y^{i,N}(s),\rho^{Y,N}_{s}\right)\right|^{2} \varphi(dv)ds \right] \\
    &~ +2\mathbb{E} \left[\int_{0}^{t} \left\langle X^{i,N}(s)-Y^{i,N}(s), f\left(s,Y^{i,N}(s),\rho^{Y,N}_{s}\right)-\Gamma_{1}\left(f\left(\tau_{n}(s),Y^{i,N}\left(\tau_{n}(s)\right),\rho^{Y,N}_{\tau_{n}(s)}\right),\Delta  t\right) \right\rangle ds \right]  \\
&~ +\underbrace{ \sum_{j=1}^{m} C \mathbb{E} \left[\int_{0}^{t}\left|g_{j}\left(s,Y^{i,N}(s),\rho^{Y,N}_{s}\right)-\Gamma_{2}\left(g_{j}\left(\tau_{n}(s),Y^{i,N}\left(\tau_{n}(s)\right),\rho^{Y,N}_{\tau_{n}(s)}\right),\Delta  t\right)\right|^{2} ds \right]}_{I_1}  \\
    &~+ \underbrace{C\mathbb{E} \left[\int_{0}^{t}\int_{\mathcal{E}} \left|h\left(s,Y^{i,N}(s),\rho^{Y,N}_{s},v\right)-\Gamma_{3}\left(h\left(\tau_{n}(s),Y^{i,N}\left(\tau_{n}(s)\right),\rho^{Y,N}_{\tau_{n}(s)},v\right),\Delta  t\right)\right|^{2} \varphi(dv)ds  \right]}_{I_2}.
\end{align*}
Applying the Cauchy–Schwarz inequality and Assumption \ref{ass:Enhanced-Coupled-mono-condi}, we arrive that
\begin{align*}
    &\mathbb{E}\left[\left|X^{i,N}(t)-Y^{i,N}(t)\right|^{2} \right]  \\
    \le &~ C\mathbb{E} \left[\int_{0}^{t} \left|X(s)^{i,N}-Y^{i,N}(s)\right|^{2} ds\right] + C\mathbb{E} \left[\int_{0}^{t} \mathbb{W}_{2}^{2}\left(\rho^{X,N}_{s},\rho^{Y,N}_{s}\right)ds\right] \\
    &~ +C\mathbb{E} \left[\int_{0}^{t} \left| f\left(s,Y^{i,N}(s),\rho^{Y,N}_{s}\right)-\Gamma_{1}\left(f\left(\tau_{n}(s),Y^{i,N}\left(\tau_{n}(s)\right),\rho^{Y,N}_{\tau_{n}(s)}\right),\Delta  t\right) \right|^{2} ds \right]  \\
 &~+ I_{1}+I_2.
\end{align*}
By using 
$$
\mathbb{W}_2^2\left(\rho^{X, N}_{s}, \rho^{Y, N}_{s}\right) \leq \frac{1}{N} \sum_{i=1}^N \left|X^{i, N}(s)-Y^{i, N}(s)\right|^2,
$$
and Lemma \ref{lem:Difference-of-coefficients}, we obtain
\begin{align*}
    \mathbb{E}\left[\left|X^{i,N}(t)-Y^{i,N}(t)\right|^{2} \right]\le &~C\mathbb{E} \left[\int_{0}^{t} \left|X^{i,N}(s)-Y^{i,N}(s)\right|^{2} ds\right] + C\mathbb{E} \left[\int_{0}^{t} \left(\frac{1}{N}\sum_{i=1}^{N}\left|X^{i,N}(s)-Y^{i,N}(s)\right|^{2}\right)ds\right] \\
    &~+C\int_{0}^{t} \Delta t ^{\frac{2}{2+\varepsilon}} \left(1+\left(\mathbb{E}\left[\left|Y_{0}^{i,N}\right|^{2\bar{p}}\right]\right)^{\beta}\right) ds     \\
    \le &~C\Delta t ^{\frac{2}{2+\varepsilon}} \left(1+\left(\mathbb{E}\left[\left|Y_{0}^{i,N}\right|^{2\bar{p}}\right]\right)^{\beta}\right) +  C \int_{0}^{t}\sup_{i\in\mathcal{I}_{N}} \mathbb{E} \left[\left|X^{i,N}(s)-Y^{i,N}(s)\right|^{2}\right] ds.
\end{align*}
Therefore, we get that for all $t \in [0,\mathcal{T}]$,
\begin{equation*}
\begin{aligned}
& \sup _{i\in \mathcal{I}_{N}} \sup _{r \in[0, t]} \mathbb{E}\left[\left|X^{i, N}(r)-Y^{i, N}(r)\right|^{2 }\right] \\
\leq & C \int_0^t \sup _{i\in \mathcal{I}_{N}} \sup _{r \in[0, s]} \mathbb{E}\left[\left|X^{i, N}(r)-Y^{i, N}(r)\right|^{2}\right] d s+C \Delta t^{\frac{2}{2+\varepsilon}} \left(1+\left(\mathbb{E}\left[\left|Y_{0}^{i,N}\right|^{2\bar{p}}\right]\right)^{\beta}\right).
\end{aligned}
\end{equation*}
Combining the preceding estimate with Gronwall's inequality completes the proof.
\qed

\section{Verification of asumptions in Section \ref{sec:examples} }
\label{verfivation-examples}

\textbf{Verification of Example \ref{ex:tanh}:} 
Based on $|\tanh(y)| \le |y|$ and $|\tanh(y)| \le 1 $, it holds that
$$
\left|\Gamma_{l}\left(F_{l}^{\mathfrak{i}}(\bm{Y}),\Delta t\right)\right| \le \left|F_{l}^{\mathfrak{i}}(\bm{Y})\right|, \quad  \left|\Gamma_{l}\left(F_{l}^{\mathfrak{i}}(\bm{Y}),\Delta t\right)\right| \le \Delta t ^{-1},~l=1,2,~\mathfrak{i}=0,1,\cdots,m,
$$
$$
\left|\Gamma_{3}\left(F_{3}^{m+1}(\bm{\bar{Y}}),\Delta t\right)\right|  \le  \left|F_{3}^{m+1}(\bm{\bar{Y}})\right| , \quad  \left|\Gamma_{3}\left(F_{3}^{m+1}(\bm{\bar{Y}}),\Delta t\right)\right| \le \Delta t ^{-1}.
$$
This implies that Assumption \ref{ass:Gamma-control-conditions} is valid with $\alpha_{1}=\alpha_{2}=\alpha_{3}=1$. Then applying the property of the hyperbolic tangent function that for any $0\le \theta \le 1$, 
 $|y-\tanh(y)|\le  |y|^{3-2\theta}$, we arrive at
$$
\left|\Gamma_{l}(F_{l}^{\mathfrak{i}}(\bm{Y}),\Delta t)-F_{l}^{\mathfrak{i}}(\bm{Y})\right|  =\Delta t^{-1} \left|\tanh\left(\Delta t F_{l}^{\mathfrak{i}}(\bm{Y})\right)-\Delta t F_{l}^{\mathfrak{i}}(\bm{Y})\right|   \le \Delta t |F_{l}^{\mathfrak{i}}(\bm{Y})|^{2},~l=1,2,~\mathfrak{i}=0,1,\cdots,m,
$$
\begin{align*}
\int_{\mathcal{E}}\left|\Gamma_{3}(F_{3}^{m+1}(\bm{\bar{Y}}),\Delta t)-F_{3}^{m+1}(\bm{\bar{Y}})\right|^{2} \varphi(dv) &~\le \int_{\mathcal{E}} \Delta t ^{-2}\left|\tanh(\Delta t F_{3}^{m+1}(\bm{\bar{Y}}))-\Delta t F_{3}^{m+1}(\bm{\bar{Y}})\right|^{2}\varphi(dv) \\
&~\le \Delta t ^{2} \int_{\mathcal{E}} |F_{3}^{m+1}(\bm{\bar{Y}})|^{4}\, \varphi(dv),
\end{align*}
where $\theta =0.5$. This implies that Assumption \ref{ass:Coefficient-comparison-conditions-of-Gamma1-Gamma3} is fulfilled  with $\delta_{1}=\delta_{2}=\delta_{3}=1$ and $\gamma_{1}=\gamma_{2}=\gamma_{3}=2$.

\textbf{Verification of Example \ref{ex:tame}:}
It is straightforward to verify that  
$$
|\Gamma_{l}\left(F_{l}^{\mathfrak{i}}(\bm{Y}),\Delta t\right)| \le \left|F_{l}^{\mathfrak{i}}(\bm{Y})\right|, \quad  |\Gamma_{l}\left(F_{l}^{\mathfrak{i}}(\bm{Y}),\Delta t \right)| \le \frac{\left|F_{l}^{\mathfrak{i}}(\bm{Y})\right|}{\Delta t\left|F_{l}^{\mathfrak{i}}\left(\bm{Y}\right)\right|} \le \Delta t^{-1}, ~l=1,2,~\mathfrak{i}=0,1,\cdots,m,
$$
$$
|\Gamma_{3}\left(F_{3}^{m+1}(\bm{\bar{Y}}),\Delta t\right)| \le \left|F_{3}^{m+1}(\bm{\bar{Y}})\right|, \quad  |\Gamma_{3}\left(F_{3}^{m+1}(\bm{\bar{Y}}),\Delta t \right)| \le \frac{\left|F_{3}^{m+1}(\bm{\bar{Y}})\right|}{\Delta t\left|F_{3}^{m+1}\left(\bm{\bar{Y}}\right)\right|} \le \Delta t^{-1},~
$$
and thus Assumption \ref{ass:Gamma-control-conditions} is  fulfilled with $\alpha_{1}=\alpha_{2}=\alpha_{3}=1$.  Moreover, by \eqref{ex:tamed-euler-scheme}, we show
$$
\left|\Gamma_{l}(F_{l}^{\mathfrak{i}}(\bm{Y}),\Delta t)-F_{l}^{\mathfrak{i}}(\bm{Y})\right|^{2}  \le  \frac{|F_{l}^{\mathfrak{i}}(\bm{Y})|^{2}\Delta t}{1+\Delta t |F_{l}^{\mathfrak{i}}(\bm{Y})|}  \le \Delta t  |F_{l}^{\mathfrak{i}}(\bm{Y})|^{2}, ~l=1,2,~\mathfrak{i}=0,1,\cdots,m,
$$
$$
\int_{\mathcal{E}}\left|\Gamma_{3}(F_{3}^{m+1}(\bm{\bar{Y}}),\Delta t)-F_{3}^{m+1}(\bm{\bar{Y}})\right|^{2} \varphi(dv) \le \int_{\mathcal{E}} \frac{|F_{3}^{m+1}(\bm{\bar{Y}})|^{4}\Delta t^{2}}{1+\Delta t |F_{3}^{m+1}(\bm{\bar{Y}})|} \varphi(dv) \le \Delta t ^{2} \int_{\mathcal{E}} |F_{3}^{m+1}(\bm{\bar{Y}})|^{4}\, \varphi(dv).
$$
i.e. Assumption \ref{ass:Coefficient-comparison-conditions-of-Gamma1-Gamma3} is satisfied with $\delta_{1}=\delta_{2}=\delta_{3}=1,$ and $   
 \gamma_{1}=\gamma_{2}=\gamma_{3}=2$. 

\section{Verification of assumptions in Section \ref{sec:Numerical-Experiments}}
\label{verfivation-numerical-examples}
\textbf{Verification of Example \ref{exam:3/2-volatility-model}: }
It can be observed in \eqref{exam:Num-3/2-Volatility}, $f(t,y,\rho)=a_{1}\left(y(a_{2}-|y|)+\int_{\mathbb{R}}y\rho(dy)\right)$, $g(t,y,\rho) = b\left(|y|^{\frac{3}{2}}+\int_{\mathbb{R}}y\rho(dy)\right)$,
$h(t,y,\rho,v) = c\left(1-y-\int_{\mathbb{R}}y\rho(dy)\right)$, where $a_{1}=6, a_{2}=2, b= -0.1, c = 1$. 
Since Assumption \ref{ass:Enhanced-Coupled-mono-condi} with $\eta=1.5$ is a reinforcement of Assumptions \ref{ass:Coupled-mono-condi}, we will only verify Assumption \ref{ass:Enhanced-Coupled-mono-condi} and Assumption \ref{ass:Coupled-mono-condi} can also be verified. By the elementary inequality, the Young inequality, the mean value theorem and $-4a_{1}+\frac{9}{2} \eta b^{2}=-23.9325<0$, we get 
\begin{align*}
&2\big\langle y-\bar{y},f(t,y,\rho)-f(t,\bar{y},\bar{\rho})\big\rangle + \eta\left|g(t,y,\rho)-g(t,\bar{y},\bar{\rho})\right|^2 + \eta \int_\mathcal{E}\left|h(t,y, \rho, v)-h(t,\bar{y}, \bar{\rho},v)\right|^2 \varphi(d v)    \\
= &~2(y-\bar{y}) \left[a_{1}a_{2}(y-\bar{y})-a_{1}(y|y|-\bar{y}|\bar{y}|)+a_{1}\int_{\mathbb{R}} y~\rho(dy)-a_{1}\int_{\mathbb{R}} \bar{y}~\bar{\rho}(d\bar{y})\right] \\
&~+\eta b^{2}\left||y|^{\frac{3}{2}}-|\bar{y}|^{\frac{3}{2}}+\int_{\mathbb{R}} y~\rho(dy)-\int_{\mathbb{R}} \bar{y}~\bar{\rho}(d\bar{y})\right|^{2} + \eta c^{2}\int_{\mathcal{E}}\left|\bar{y}-y+\int_{\mathbb{R}} \bar{y}~\bar{\rho}(d\bar{y})-\int_{\mathbb{R}} y~\rho(dy)\right|^{2} \varphi(dv)  \\
\le &~(a_{1}+2a_{1}a_{2}+2\eta c^{2}\lambda)|y-\bar{y}|^{2} +
\left(-4a_{1}+\frac{9}{2} \eta b^{2}\right)|y-\bar{y}|^{2} \int_{0}^{1}|\xi|~dr \\
&~+(a_{1}+2\eta b^{2}+2\eta c^{2}\lambda)\left[\int_{\mathbb{R}} y~\rho(dy)-\int_{\mathbb{R}} \bar{y}~\bar{\rho}(d\bar{y})\right]^{2}  \\
\le &~(a_{1}+2a_{1}a_{2}+2\eta b^{2}+2\eta c^{2}\lambda)\left(|y-\bar{y}|^{2} + \mathbb{W}_{2}^{2}(\rho,\bar{\rho})\right),
\end{align*}
where $\xi:= \bar{y}+r(y-\bar{y})$. Therefore, $C=a_{1}+2a_{1}a_{2}+2\eta b^{2}+2\eta c^{2}\lambda=36.03$ in the Assumption \ref{ass:Enhanced-Coupled-mono-condi}.
Next, we will verify the Assumption \ref{ass:Coercivity-condition-Chen}, here $\theta =1$ and $\bar{p}=297$ is sufficiently large for our Assumption. Using the elementary inequality, the Young inequality, the $\rm H\ddot{o}lder$ inequality and $-a_{1}+b^{2}(2\bar{p}-1)=-0.07<0$, we infer
\begin{align*}
      & 2\bar{p}|y|^{2\bar{p}-2}\big\langle y, f(t,y,\rho)\big\rangle + \bar{p}(2\bar{p}-1)|y|^{2\bar{p}-2}|g(t,y,\rho)|^2 +(1+(2\bar{p}-2)\theta)\int_{\mathcal{E}}|h(t,y,\rho,v)|^{2\bar{p}} \varphi(d v)\\
    =&~2\bar{p}|y|^{2\bar{p}-2}\left\langle y, a_{1}a_{2}y-a_{1}y|y|+a_{1}\int_{\mathbb{R}}y\,\rho(dy) \right\rangle + \bar{p}(2\bar{p}-1)|y|^{2\bar{p}-2}b^{2} \left||y|^{\frac{3}{2}}+\int_{\mathbb{R}}y\,\rho(dy)\right|^{2}  \\
     &+(1+(2\bar{p}-2)\theta)\int_{\mathcal{E}}c^{2\bar{p}}\left|(1-y)-\int_{\mathbb{R}}y\,\rho(dy)\right|^{2\bar{p}}\varphi(d v) \\
     \le &~ \left(2a_{1}a_{2}\bar{p}+a_{1}\bar{p}+a_{1}(\bar{p}-1)+(2\bar{p}-1)c^{2\bar{p}} 2^{4\bar{p}-2}\lambda \right)\left(1+|y|^{2\bar{p}}\right) + 2\bar{p}\left(-a_{1}+b^{2}(2\bar{p}-1)\right)|y|^{2\bar{p}+1}\\
     &~+\left(-2a_{1}\bar{p}+2b^{2}\bar{p}(2\bar{p}-1)\right)\left(\int_{\mathbb{R}}y\,\rho(dy)\right) ^{2\bar{p}} \\
     \le &~\left(2a_{1}a_{2}\bar{p}+a_{1}\bar{p}+a_{1}(\bar{p}-1)+(2\bar{p}-1)c^{2\bar{p}} 2^{4\bar{p}-2}\lambda \right)\left(1+|y|^{2\bar{p}}+\mathbb{W}_{2}^{2\bar{p}}(\rho,\delta_{0})\right).
\end{align*}
Finally, we verify the Assumption \ref{ass:poly-growth-coeff-a}
\begin{align*}
 &\left|f(t,y,\rho)-f(t,\bar{y},\bar{\rho})\right| \\
 = &~\left|a_{1}a_{2}(y-\bar{y})-a_{1}(y|y|-\bar{y}|\bar{y}|)+\int_{\mathbb{R}}y~\rho(dy)-\int_{\mathbb{R}}\bar{y}~\bar{\rho}(d\bar{y})\right| \\
 \le &~a_{1}a_{2}|y-\bar{y}| + 2a_{1}|y-\bar{y}|\int_{0}^{1} |\xi| ~dr + \mathbb{W}_{2}(\rho,\bar{\rho}) \\
 \le &~2a_{1} \left((1+|y|+|\bar{y}|)|y-\bar{y}|+\mathbb{W}_{2}(\rho,\bar{\rho})\right),
\end{align*}
with $C=2a_{1}=12$, $\gamma=1$ and here $\xi$ as above.  

\textbf{Verification of Example \ref{exam:double-well-model}:}
We can observe in \eqref{exam:num-double-well model}, $f(t,y,\rho)=d_{1}\left(y(1-y^{2})+\int_{\mathbb{R}} y ~\rho(dy)\right)$, $g(t,y,\rho)=d_{2} \left(1-y^{2}-\int_{\mathbb{R}} y ~\rho(dy)\right)$ and $h(t,y,\rho,v)=d_{3}\left(y\ln(1+y^{2})+\int_{\mathbb{R}} y ~\rho(dy)\right)$ with $d_{1}=66, d_{2}=0.18, d_{3}=0.0006$. Similarly, we first verify Assumption \ref{ass:Enhanced-Coupled-mono-condi} with $\eta=1.5$ and thus deduce Assumption \ref{ass:Coupled-mono-condi}. Using the Young inequality, the mean-value theorem, $\ln(1+y^{2}) \le |y| $ and $-6d_{1}+8\eta d_{2}^{2}+2\lambda \eta d_{3}^{2}\approx -395.5668<0$, we arrive
\begin{align*}
&2\big\langle y-\bar{y},f(t,y,\rho)-f(t,\bar{y},\bar{\rho})\big\rangle + \eta\left|g(t,y,\rho)-g(t,\bar{y},\bar{\rho})\right|^2 + \eta \int_\mathcal{E}\left|h(t,y, \rho, v)-h(t,\bar{y}, \bar{\rho},v)\right|^2 \varphi(d v)    \\
=&~2d_{1}(y-\bar{y})\left[y(1-y^{2})+\int_{\mathbb{R}}y~\rho(dy)-\bar{y}(1-\bar{y}^{2})+\int_{\mathbb{R}} \bar{y} ~\bar{\rho}(d\bar{y})\right]  \\
&~+ \eta \left|d_{2}\left(1-y^{2}-\int_{\mathbb{R}}y~\rho(dy)\right)-d_{2}\left(1-\bar{y}^{2}-\int_{\mathbb{R}}\bar{y}~\bar{\rho}(d\bar{y})\right)\right|^{2} \\
&~+\lambda \eta \left|d_{3}\left(y\ln(1+y^{2})+\int_{\mathbb{R}}y~\rho(dy)\right)-d_{3}\left(\bar{y}\ln(1+\bar{y}^{2})+\int_{\mathbb{R}}\bar{y}~\bar{\rho}(d\bar{y})\right)\right|^{2} \\
\le &~3d_{1}\left|y-\bar{y}\right|^{2} -6d_{1}\left|y-\bar{y}\right|^{2} \int_{0}^{1} \left|\xi\right|^{2}\;dr +  (d_{1}+2\eta d_{2}^{2}+2\lambda \eta d_{3}^{2})\left(\int_{\mathbb{R}}y~\rho(dy)-\int_{\mathbb{R}}\bar{y}~\bar{\rho}(d\bar{y})\right)^{2} + 8\eta d_{2}^{2} \left|y-\bar{y}\right|^{2}\int_{0}^{1} \left|\xi\right|^{2} dr \\
&~+ 2\lambda \eta d_{3}^{2} \left|y-\bar{y}\right|^{2}\int_{0}^{1} \left|\ln^{2}(1+\xi^{2})\right|^{2} dr  + 4\lambda \eta d_{3}^{2} \left|y-\bar{y}\right|^{2} \int_{0}^{1} \left(\frac{\xi^{2}}{1+\xi^{2}}\right) dr \\
\le &~ (3d_{1}+4\lambda\eta d_{3}^{2}) \left|y-\bar{y}\right|^{2} +\left(-6d_{1}+8\eta d_{2}^{2}+2\lambda \eta d_{3}^{2}\right)\left|y-\bar{y}\right|^{2}\int_{0}^{1} \left|\xi\right|^{2} dr + (d_{1}+2\eta d_{2}^{2}+ 2\lambda \eta d_{3}^{2}) ~\mathbb{W}_2^2(\rho, \bar{\rho}) \\
\le &~C\left(|y-\bar{y}|^{2}+ \mathbb{W}_2^2(\rho, \bar{\rho})\right),
\end{align*}
where $C=3d_{1}+4\lambda\eta d_{2}^{2}=198.4332$ and  $\xi=\bar{y}+r(y-\bar{y})$ as before.
Next, we will verify the Assumption \ref{ass:Coercivity-condition-Chen}. Using the Cauchy-Schwarz inequality, the Young inequality, and $|\ln(1+y^{2})|\le \bar{p}|y|^{\frac{1}{\bar{p}}}$, one can obtain
\begin{align*}
      & 2\bar{p}|y|^{2\bar{p}-2}\big\langle y, f(t,y,\rho)\big\rangle + \bar{p}(2\bar{p}-1)|y|^{2\bar{p}-2}|g(t,y,\rho)|^2 +(1+(2\bar{p}-2)\theta)\int_{\mathcal{E}}|h(t,y,\rho,v)|^{2\bar{p}} \varphi(d v)\\
      =&~2\bar{p} d_{1} |y|^{2\bar{p}-2} \left\langle y, y-y^{3} + \int_{\mathbb{R}} y\;\rho(dy) \right\rangle + \bar{p} (2\bar{p}-1) |y|^{2\bar{p}-2} \left|d_{2}\left(1-y^{2}-\int_{\mathbb{R}}y\;\rho(dy)\right)\right|^{2} \\
      &~+ (1+(2\bar{p}-2)\theta) \int_{\mathcal{E}} \left|d_{3}\left(y\ln(1+y^{2})+\int_{\mathbb{R}}y\;\rho(dy)\right)\right|^{2\bar{p}} \varphi(dv) \\
      \le &~\left(-2d_{1}\bar{p}+d_{2}^{2}(1+\theta_{1})\bar{p}(2\bar{p}-1)+(1+(2\bar{p}-2)\theta)\lambda d_{3}^{2\bar{p}}(1+\theta_{2})\bar{p}^{2\bar{p}}\right)|y|^{2\bar{p}+2} \\
      &~+\left((4\bar{p}-1)d_{1}+2d_{2}^{2}\bar{p}(2\bar{p}-1)+(2\bar{p}-1)d_{2}^{2}\left(2+\frac{1}{\theta_{1}}\right)(\bar{p}-1)\right)\left|y\right|^{2\bar{p}} \\
      &~+\left(d_{1}+d_{2}^{2}(2\bar{p}-1)\left(2+\frac{1}{\theta_{1}}\right)+\left(1+(2\bar{p}-2)\theta\right)\lambda d_{3}^{2\bar{p}}\cdot \left(1+\sum_{i=1}^{2\bar{p}-1}\frac{i}{2\bar{p}}\left(C_{2\bar{p}}^{i}\right)^{\frac{2\bar{p}}{i}}\left(\frac{2\bar{p}\theta_{2}}{(2\bar{p}-i)(2\bar{p}-1)} \right)^{-\frac{2\bar{p}-i}{i}}\right)\right) \\
      &~\qquad \cdot \left(\int_{\mathbb{R}}y\;\rho(dy)\right)^{2\bar{p}}  + 2\bar{p} (2\bar{p}-1)d_{2}^{2} \\
      \le &~C\left(1+|y|^{2\bar{p}}+\mathbb{W}_{2}^{2\bar{p}}(\rho,\delta_{0})\right).
\end{align*}
where $-2d_{1}\bar{p}+d_{2}^{2}(1+\theta_{1})\bar{p}(2\bar{p}-1)+(1+(2\bar{p}-2)\theta)\lambda d_{3}^{2\bar{p}}(1+\theta_{2})\bar{p}^{2\bar{p}}\approx -9.6736<0$, and we set $\theta=\theta_{1}=\theta_{2}=0.1144$ and $\bar{p}=1641$.
Finally, we will verify the Assumption \ref{ass:poly-growth-coeff-a}. By the elementary inequality, the mean-value theorem and the $\rm H\ddot{o}lder$ inequality, one can obtain
\begin{align*}
    &\left|f(t,y,\rho)-f(t,\bar{y},\rho)\right|  \\
    =&~ \left|d_{1}\left(y-y^{3}+\int_{\mathbb{R}}y~\rho(dy)-\bar{y}-\bar{y}^{3}-\int_{\mathbb{R}}\bar{y}~\bar{\rho}(d\bar{y})\right)\right|  \\
    \le &~|d_{1}| \left|y-\bar{y}\right| + |d_{1}|\left|y^{3}-\bar{y}^{3}\right| + |d_{1}| \left|\int_{\mathbb{R}} y~\rho(dy)-\int_{\mathbb{R}}\bar{y}~\bar{\rho}(d\bar{y})\right| \\
    \le &~|d_{1}|\left|y-\bar{y}\right| + 3|d_{1}| \left|y-\bar{y}\right| \int_{0}^{1} \left|\xi\right|^{2} dr + |d_{1} |\left|\int_{\mathbb{R}}y~\rho(dy)-\int_{\mathbb{R}}\bar{y}~\bar{\rho}(d\bar{y})\right| \\
    \le &~C\left(\left(1+|y|^{2}+|\bar{y}|^{2}\right)|y-\bar{y}|+\mathbb{W}_{2}(\rho,\bar{\rho})\right),
\end{align*}
where $\xi$ as before, and $C=6d_{1} = 396$, and $ \gamma = 2$ in Assumption \ref{ass:poly-growth-coeff-a}.

\end{document}